\documentclass[10pt]{amsart}

 \usepackage{amsmath}
 \usepackage{amsthm}
 \usepackage{MnSymbol}

\newcommand{\N}{\mathbb{N}}
\newcommand{\R}{\mathbb{R}}
\newcommand{\C}{\mathbb{C}}
\newcommand{\D}{\mathbb{D}}

\newcommand{\T}{\mathbb{T}}

\newtheorem{theorem}{Theorem}
\newtheorem{lemma}[theorem]{Lemma}
\newtheorem{corollary}[theorem]{Corollary}

\theoremstyle{definition}

\newtheorem{problem}{Problem}

\renewcommand{\Re}{\operatorname{Re}\,}
\renewcommand{\Im}{\operatorname{Im}\,}

\title[An analog of Suffridge's convolution theorem]{Suffridge's
  convolution theorem for polynomials with zeros in the unit disk}

\author{Martin Lamprecht} 

\address{Department of Computer Science and Engineering\\
European University of Cyprus\\
Diogenous Str. 6, Engomi, P.O. Box: 22006, 1516 Nicosia, Cyprus}
\email{m.lamprecht@euc.ac.cy}

\address{Department of Mathematics and Statistics\\
University of Cyprus\\
P.O. Box: 20537, 1678 Nicosia, Cyprus}
\email{lmartin@ucy.ac.cy}

\keywords{Suffridge polynomials, Grace's theorem}

\subjclass[2010]{30C10, 30C15}

\begin{document}

\maketitle
\begin{abstract}
  In 1976 Suffridge proved an intruiging theorem regarding the convolution of
  polynomials with zeros only on the unit circle. His result generalizes a
  special case of the fundamental Grace-Szeg\"o convolution theorem, but so
  far it is an open problem whether there is a Suffridge-like extension of the
  general Grace-Szeg\"o convolution theorem. In this paper we try to approach
  this question from two different directions: First, we show that Suffridge's
  convolution theorem holds for a certain class of polynomials with zeros in
  the unit disk and thus obtain an extension of one further special case of
  the Grace-Szeg\"o convolution theorem. Second, we present non-circular zero
  domains which stay invariant under the Grace-Szeg\"o convolution hoping
  that this will lead to further analogs of Suffridge's convolution theorem. 
\end{abstract}

\section{Introduction}
\label{sec:introduction}

In 1922 Szeg\"o{} \cite{szegoe} found the following rephrasing of a theorem
of Grace \cite{grace} from 1902 regarding the zeros of apolar polynomials.

\begin{theorem}[Grace-Szeg\"o]
  \label{sec:introduction-1}
  Let 
  \begin{equation*}
    F(z) = \sum_{k=0}^{n} {n \choose k} a_{k} z^{k}\quad\mbox{and} \quad G(z) =
    \sum_{k=0}^{n} {n \choose k} b_{k} z^{k} 
  \end{equation*}
  be polynomials of degree $n\in\N$ and suppose $K\subset \C$ is an
  open or closed disk or half-plane that contains all zeros of
  $F$. Then each zero $\gamma$ of 
  \begin{equation*}
    F*_{GS} G(z) := \sum_{k=0}^{n} {n\choose k} a_{k} b_{k} z^{k}
  \end{equation*}
  is of the form $\gamma = -\alpha \beta$ with $\alpha
  \in K$ and $G(\beta)=0$. If $G(0)\neq 0$, this also holds when $K$
  is the open or closed exterior of a disk.
\end{theorem}

This result is usually called the \emph{Grace-Szeg\"o convolution theorem} (or
simply \emph{Grace's theorem}), since the \emph{convolution} or \emph{Hadamard
  product} of two functions $f(z)=\sum_{k=0}^{\infty} a_{k} z^{k}$ and
$g(z)=\sum_{k=0}^{\infty} b_{k} z^{k}$, analytic in a neighborhood of the
origin, is given by
\begin{equation*}
  (f*g)(z) := \sum_{k=0}^{\infty} a_{k}b_{k} z^{k}.
\end{equation*}
The weighted convolution $F*_{GS} G$ appearing in Grace's theorem is called the
\emph{Grace-Szeg\"o convolution} of $F$ and $G$.

The Grace-Szeg\"o convolution theorem, together with its many equivalent
forms (cf. \cite[Ch. 3]{rahman}, \cite[Ch. 5]{sheil}), is perhaps the single
most important result regarding the zero location of complex polynomials. For
instance, since
\begin{equation}
  \label{eq:1}
  F*_{GS} z(1+z)^{n-1} = z\frac{F'(z)}{n}
\end{equation}
for every polynomial $F$ of degree $n$, it is easy to see that Grace's
theorem implies the following fundamental fact.

\begin{theorem}[Gau\ss-Lucas]
  The convex hull of the zeros of a polynomial $F$ contains all zeros of
  $F'$.
\end{theorem}

More generally, Grace's theorem can be used to obtain information about the
relation between zeros and critical points of polynomials. It therefore seems
reasonable to hope that a better understanding of Grace's theorem will lead
to progress on long-standing open problems such as the conjectures of Sendov
or Smale (cf. \cite[Ch. 7]{rahman}, \cite[Ch. 6, 10.4]{sheil}).

In \cite{ruschsheil73} Ruscheweyh and Sheil-Small were able to settle a
famous conjecture of P\'olya and Schoenberg \cite{polschoe58} regarding the
convolution invariance of schlicht convex mappings. Shortly afterwards
Suffridge \cite{suffridge} found an intruiging extension of a special case of
Grace's theorem which enabled him to generalize Ruscheweyh's and
Sheil-Small's theorem. Subsequently, more extensions of Suffridge's theorem
and other special cases of Grace's theorem were found by Ruscheweyh, Salinas,
Sheil-Small, and the author
(cf. \cite{lam12,lam11,rusch77,rusch82,ruschsal08,ruschsal09,
  ruschsalsug09,sheil78,sheil}). The extensions of Grace's theorem found in
these papers show strong similarities; it thus seems very likely that
there should be a generalization of Grace's theorem which unifies all partial
extensions that have been discovered until now.

In this paper we will present two additional extensions of Grace's theorem of
a spirit similar to the one exhibited in
\cite{lam12,lam11,rusch77,rusch82,ruschsal08,ruschsal09,
  ruschsalsug09,sheil78,sheil}. We hope that this will be of help in finding
the desired unified extension of Grace's theorem.


\subsection{The main result: An extension of Suffridge's convolution theorem to
  polynomials with zeros in the unit disk}
\label{sec:an-extens-suffr}

Denote by $\pi_{n}(\Omega)$ and $\pi_{\leq n} (\Omega)$ the sets of
polynomials of degree $n$ and $\leq n$, respectively, which have zeros only
in the set $\Omega\subseteq \C$. For certain $\Omega$ Grace's theorem leads
to interesting invariance results concerning the classes $\pi_{n}(\Omega)$ or
$\pi_{\leq n}(\Omega)$. For instance, if $F\in\pi_{n}(\overline{\D})$ and
$G\in\pi_{n}(\D)$, where $\D:=\{z\in\C:|z|<1\}$ denotes the open unit disk, then
it follows from Grace's theorem that $F*_{GS} G\in\pi_{n}(\D)$. On the other
hand, if $G$ of degree $n$ is such that $F*_{GS} G\in\pi_{n}(\D)$ for all
$F\in\pi_{n}(\overline{\D})$, then the special choice $F = (1+z)^{n}$ yields
$G\in\pi_{n}(\D)$. Hence, the following is true.

\begin{corollary}
  \label{sec:introduction-3}
  Let $G$ be a polynomial of degree $n$. Then $F*_{GS}G\in\pi_{n}(\D)$ for all
  $F\in\pi_{n}(\overline{\D})$ if, and only if, $G\in\pi_{n}(\D)$.
\end{corollary}

The same holds with $\overline{\D}$ and $\D$ replaced by $\C\setminus\D$ and
$\C\setminus\overline{\D}$, respectively. Combining these two special cases of
Grace's theorem one obtains the following result concerning polynomials with
zeros only on the unit circle $\T:=\{z\in\C:|z|=1\}$.

\begin{corollary}
  \label{sec:introduction-4}
  Let $G$ be a polynomial of degree $n$. Then $F*_{GS}G\in\pi_{n}(\T)$ for all
  $F\in\pi_{n}(\T)$ if, and only if, $G\in\pi_{n}(\T)$.
\end{corollary}

In \cite{suffridge} Suffridge found an intruiging extension of Corollary
\ref{sec:introduction-4}. In order to state his results, recall that the
\emph{$q$-binomial} or \emph{Gaussian central coefficients ${n \brack k}_{q}$}
are defined by (cf. \cite{kaccheung2002} or \cite[Ch. 10]{andaskroy99} as
general references regarding $q$-binomial coefficients)
\begin{equation}
  \label{eq:16}
  R_{n}(q;z):=\sum_{k=0}^{n} q^{k(k-1)/2} {n \brack k}_{q} z^{k} 
  := \prod_{j=1}^{n}(1+q^{j-1}z), \qquad n\in\N,\;q\in\C.
\end{equation}
For reasons of brevity, in the following, for $\lambda\in[0,\frac{2\pi}{n}]$,
we use the notation
\begin{equation}
  \label{eq:2}
  Q_{n}(\lambda;z):= \sum_{k=0}^{n} C_{k}^{(n)}(\lambda) z^{k}:= 
  \prod_{j=1}^{n}(1+e^{i(2j-n-1)\lambda/2}z),
\end{equation}
such that
\begin{equation}
  \label{eq:17}
  Q_{n}(\lambda;z) = R_{n}(e^{i\lambda};e^{-i(n+1)\lambda/2}z)  
  \quad \mbox{and} \quad
  C_{k}^{(n)}(\lambda):= e^{ik(k-n-2)\lambda/2} {n \brack k}_{e^{i\lambda}}.
\end{equation}
Then
\begin{equation}
  \label{eq:6}
  Q_{n}(0;z)=(1+z)^{n}, \quad \mbox{and thus} \quad C_{k}^{(n)}(0)= {n\choose k},
  \qquad k=0,\ldots,n,
\end{equation}
and
\begin{equation}
  \label{eq:43}
  Q_{n}({\textstyle \frac{2\pi}{n}};z)=1+z^{n}.
\end{equation}
We call a polynomial of the form $F(z)=a\, Q_{n}(\lambda;b z)$, with
$a\in\C\setminus\{0\}$ and $b\in\T$, a \emph{$\lambda$-extremal polynomial}.  It is
well known (cf. \cite{suffridge} or \cite[Ch. 7]{sheil}) that
\begin{equation}
  \label{eq:3}
  C_{k}^{(n)}(\lambda) = \frac{\prod_{j=1}^{n}\sin (j\lambda/2)}
  {\prod_{j=1}^{k}\sin (j\lambda/2) 
    \prod_{j=1}^{n-k}\sin (j\lambda/2)} \neq 0 \quad \mbox{for} \quad 
  \lambda\in(0,{\textstyle{\frac{2\pi}{n}}}).
\end{equation}
Hence, for all $\lambda\in[0,\frac{2\pi}{n})$, every pair of polynomials $F$,
$G$ of degree $n$ can be written in the form
\begin{equation*}
  F(z)= \sum_{k=0}^{n} C_{k}^{(n)}(\lambda) a_{k} z^{k},\qquad
  G(z)= \sum_{k=0}^{n} C_{k}^{(n)}(\lambda) b_{k} z^{k}, 
\end{equation*}
and we can define 
\begin{equation}
  \label{eq:4}
  F*_{\lambda}G (z):= \sum_{k=0}^{n} C_{k}^{(n)}(\lambda) a_{k} b_{k} z^{k}. 
\end{equation}
Then, because of (\ref{eq:6}),
\begin{equation}
  \label{eq:5}
  F*_{0}G = F*_{GS}G.
\end{equation}

All zeros of $Q_{n}(\lambda;z)$ lie on $\T$ with each (except one) pair of
consecutive zeros separated by an angle of exactly $\lambda$. Suffridge
\cite{suffridge} introduced the classes $\overline{\mathcal{T}}_{n}(\lambda)$ of
polynomials in which $Q_{n}(\lambda;z)$ is the natural extremal element (note,
however, that in \cite{suffridge} the classes
$\overline{\mathcal{T}}_{n}(\lambda)$ are denoted by
$\mathcal{P}_{n}(\lambda)$). More exactly, the classes
$\overline{\mathcal{T}}_{n}(\lambda)$ are defined to consist of all polynomials
of degree $n$ that have zeros only on $\T$ with each pair of zeros separated by
an angle of at least $\lambda$. $\mathcal{T}_{n}(\lambda)$ shall denote the set
of those $F$ in $\overline{\mathcal{T}}_{n}(\lambda)$ for which every pair of
zeros is separated by an angle $>\lambda$. With these definitions we have that
$\overline{\mathcal{T}}_{n}(\frac{2\pi}{n}) =
\{a(1+bz^{n}):a\in\C\setminus\{0\}, b\in\T\}$, $\mathcal{T}_{n}(\frac{2\pi}{n})
= \emptyset$, $\overline{\mathcal{T}}_{n}(0) = \pi_{n}(\T)$, and that
$\mathcal{T}_{n}(0)$ is the set of those $f\in
\overline{\mathcal{T}}_{n}(0)$ which have only simple zeros.

In the following, for every class $\mathcal{C}_{n}(\lambda)$ of polynomials of
degree $n$ depending on $\lambda\in[0,\frac{2\pi}{n})$, we define the
\emph{pre-coefficient classes} $\mathcal{PC}_{n}(\lambda)$ to consist of all
\begin{equation*}
  f(z)=\sum_{k=0}^{n}a_{k} z^{k}\quad 
  \mbox{for which}
  \quad f*Q_{n}(\lambda;z)= 
  \sum_{k=0}^{n}C_{k}^{(n)}(\lambda)a_{k}z^{k} \in \mathcal{C}_{n}(\lambda).
\end{equation*}
 For instance,
$\mathcal{P}\overline{\mathcal{T}}_{n}(\lambda)$ is the class of those
polynomials $f$ for which $f*Q_{n}(\lambda;z)$ belongs to
$\overline{\mathcal{T}}_{n}(\lambda)$. The polynomials $a \sum_{k=0}^{n}
b^{k}z^{k}$, with $a\in\C\setminus\{0\}$, $b\in\T$, will belong to every
pre-coefficient class considered, and will be called \emph{pre-extremal
  polynomials}.

Suffridge's main results from \cite{suffridge} can now be stated as follows (see
also \cite{lam11} for a different proof, and note that (\ref{item:21}) is an
equivalent form of one of Suffridge's result from \cite{suffridge} which can be
deduced from \cite[Thm. 7.6.9]{sheil}).

\begin{theorem}[Suffridge]
  \label{sec:an-extens-suffr-1}
  Let $\lambda\in[0,\frac{2\pi}{n})$.
  \begin{enumerate}
  \item \label{item:19} If $G$ is a polynomial of degree $n$, then $F
    *_{\lambda} G \in \mathcal{T}_{n}(\lambda)$ for all $F\in
    \overline{\mathcal{T}}_{n}(\lambda)$ if, and only if,
    $G\in\mathcal{T}_{n}(\lambda)$ .
  \item \label{item:20} If $\lambda < \mu < \frac{2\pi}{n}$ and
    $f\in\mathcal{P}\overline{\mathcal{T}}_{n}(\lambda)$ is not pre-extremal,
    then $f\in\mathcal{PT}_{n}(\mu)$.
  \item \label{item:21} Let $f(z) = \sum_{k=0}^{n} a_{k}z^{k}$ be a polynomial
    of degree $n$ whose zeros lie symmetrically around $\T$. Then $f$ belongs to
    $\mathcal{P}\mathcal{T}_{n}(\lambda)$ for a $\lambda\in [0,\frac{2\pi}{n})$
    if, and only if,
    \begin{equation*}
      \frac{f(z)-a_{0}}{a_{n} z^{n} - a_{0}} \quad\mbox{maps} \quad \C
      \setminus\overline{\D} \quad n \mbox{-fold onto the half-plane} 
      \quad \Re z > \frac{1}{2}.
    \end{equation*}
  \end{enumerate}
\end{theorem}

For $\lambda = 0$ Theorem \ref{sec:an-extens-suffr-1}\ref{item:19} (which
will be called Suffridge's convolution theorem from now on) is essentially
equal to the $\T$-special case of Grace's theorem stated in Corollary
\ref{sec:introduction-4}. Because of (\ref{eq:16}) and (\ref{eq:17}),
Suffridge's convolution theorem constitutes a $q$-extension of this special
case, albeit only for $q$ of the form $q=e^{i\lambda}$.

It is so far unknown whether there is an extension of Grace's theorem which
includes Suffridge's convolution theorem as a special case.  In order to look
for such an extension, it seems a promising approach to check whether there are
Suffridge-type extensions for other special cases of Grace's theorem. In
\cite{lam12} such extensions were found for the 'half-plane cases' of Grace's
theorem (cf. Corollaries \ref{sec:introduction-5} and
\ref{sec:introduction-6} below). However, until now it was not clear how the
corresponding extension of Corollary \ref{sec:introduction-3} should look
like. This is mainly due to the fact that it is not obvious what the
'natural' analog of the zero separation condition, used to define the classes
$\overline{\mathcal{T}}_{n}(\lambda)$, should be, if one considers polynomials
with zeros in $\D$. We do not have an answer to this
question yet, but the main result of this paper is an analog of
Theorem \ref{sec:an-extens-suffr-1} for certain sets
$\overline{\mathcal{D}}_{n}(\lambda) \supset
\overline{\mathcal{T}}_{n}(\lambda)$, containing polynomials with zeros in
$\D$, whose definition is motivated as follows:

For $\lambda\in(0,\frac{2\pi}{n})$ and a polynomial $F$ of degree $\leq n$ set
\begin{equation}
  \label{eq:9}
  \begin{split}
    &F_{+}(z) := F_{+,\lambda}(z):=F(e^{i\lambda/2}z) = F*_{\lambda}
    Q_{n}(\lambda;e^{i\lambda/2}z) \quad \mbox{ and} \\
    &F_{-}(z) := F_{-,\lambda}(z):=F(e^{-i\lambda/2}z) = F*_{\lambda}
    Q_{n}(\lambda;e^{-i\lambda/2}z).
  \end{split}
\end{equation}
Using \cite[Thms. 7.2.4, 7.5.2]{sheil}, it is easy to see that the
following holds.

\begin{lemma}
  \label{sec:main-result:-an-1}
  Let $\lambda\in(0,\frac{2\pi}{n})$. A polynomial $F$ of degree $n$
  belongs to $\overline{\mathcal{T}}_{n}(\lambda)$ if, and only if,
  the rational function
  \begin{equation*}
    e^{-in\lambda/2}\frac{F_{+}(z)}{F_{-}(z)}
  \end{equation*}
  maps $\C\setminus\overline{\D}$ onto the open upper
  half-plane. Every point in the open upper half-plane has exactly $n$
  pre-images in $\C\setminus\overline{\D}$ under this mapping if, and
  only if, $F\in\mathcal{T}_{n}(\lambda)$.
\end{lemma}
 
This result gives the motivation to consider the
classes $\overline{\mathcal{D}}_{n}(\lambda)$, $\lambda\in(0,\frac{2\pi}{n})$,
of all polynomials $F$ of degree $n$ for which
\begin{equation}
  \label{eq:8}
  \Im \left(e^{-in\lambda/2} \frac{F_{+}(z)}{F_{-}(z)}\right)>0, 
  \qquad z\in\C\setminus\overline{\D}.
\end{equation}
We also define $\mathcal{D}_{n}(\lambda)$ as the union of
$\mathcal{T}_{n}(\lambda)$ with the set of all polynomials $F$ of degree $n$
for which
\begin{equation}
  \label{eq:22}
  \Im \left(e^{-in\lambda/2} \frac{F_{+}(z)}{F_{-}(z)}\right)>0, 
  \qquad z\in\C\setminus\D. 
\end{equation}
These definitions imply that $\mathcal{T}_{n}(\lambda) \subset
\mathcal{D}_{n}(\lambda)$ and $\overline{\mathcal{T}}_{n}(\lambda) \subset
\overline{\mathcal{D}}_{n}(\lambda)$. The inclusions are strict, since $z^{n}
\in \mathcal{D}_{n}(\lambda)$ for all $\lambda\in(0,\frac{2\pi}{n})$. Moreover,
it readily follows from (\ref{eq:8}) and (\ref{eq:22}) that
$\mathcal{D}_{n}(\lambda)\setminus\mathcal{T}_{n}(\lambda) \subseteq
\pi_{n}(\D)$ and $\overline{\mathcal{D}}_{n}(\lambda) \subseteq
\pi_{n}(\overline{\D})$ for all $\lambda\in(0,\frac{2\pi}{n})$.  We will explain
later (cf. Theorem \ref{sec:an-extens-suffr-4} below) why it is natural to set
\begin{equation}
  \label{eq:44}
  \begin{split}
    &\overline{\mathcal{D}}_{n}({\textstyle \frac{2\pi}{n}}):= \{a(z^{n}-b):
    a\in\C\setminus\{0\}, b\in\overline{\D}\}, \qquad 
    \mathcal{D}_{n}({\textstyle \frac{2\pi}{n}}):=\emptyset,\\
    &\overline{\mathcal{D}}_{n}(0):=\pi_{n}(\overline{\D}),\qquad\mbox{and}\qquad
    \mathcal{D}_{n}(0):=\pi_{n}(\D)\cup\mathcal{T}_{n}(0).
  \end{split}
\end{equation}

These properties of the classes $\overline{\mathcal{D}}_{n}(\lambda)$ show
that the next result is a $q$-extension of Corollary \ref{sec:introduction-3}
which also contains Theorem \ref{sec:an-extens-suffr-1} (for the definition
of the $n$-inverse $f^{*n}$ of $f$ appearing in Statement (\ref{item:3}), see
Section \ref{sec:furth-prop-char} below).

\begin{theorem}
\label{sec:an-extens-suffr-3}
  Let $\lambda\in[0,\frac{2\pi}{n})$. 
  \begin{enumerate}
  \item \label{item:2} If $G$ is a polynomial of degree $n$, then
    $F*_{\lambda}G\in\mathcal{D}_{n}(\lambda)$ for all
    $F\in\overline{\mathcal{D}}_{n}(\lambda)$ if, and only if,
    $G\in\mathcal{D}_{n}(\lambda)$.
  \item \label{item:3} If $\lambda<\mu<\frac{2\pi}{n}$ and if
    $f\in\mathcal{P}\overline{\mathcal{D}}_{n}(\lambda)$ is such that $f + \zeta
    f^{*n}$ is not pre-extremal for any $\zeta\in\T$, then
    $f\in\mathcal{PD}_{n}(\mu)$. 
  \item \label{item:4} Let $f(z) = \sum_{k=0}^{n} a_{k}z^{k}$ be of degree
    $n$ with $|a_{0}|< |a_{n}|$. There is a $\lambda\in
    [0,\frac{2\pi}{n})$ such that $f$ belongs to
    $\mathcal{PD}_{n}(\lambda)\setminus \mathcal{PT}_{n}(\lambda)$ if, and
    only if,
    \begin{equation*}
      \Re \frac{f(z)-a_{0}}{a_{n} z^{n} - a_{0}} > \frac{1}{2}, \qquad 
      z\in\C \setminus\D.
    \end{equation*}
  \end{enumerate}
\end{theorem}

Even though this is the desired extension of Suffridge's theorem to
polynomials with zeros in $\D$, Theorem \ref{sec:an-extens-suffr-3} remains
unsatisfactory, since we do not have an explicit description of the zero
location of polynomials in $\overline{\mathcal{D}}_{n}(\lambda)$. Such a
description would in particular be important for obtaining a $q$-analog of a
very useful reformulation of Grace's theorem which is due to Walsh
(cf. \cite[Thm. 3.4.1b]{rahman} or \cite[Thm. 5.2.7]{sheil}), or for finding
Suffridge-type extensions of the special cases of Grace's theorem which are
presented in Section \ref{sec:cont-conn-betw} below.

We therefore regard the following as the main open problem concerning the
classes $\overline{\mathcal{D}}_{n}(\lambda)$.

\begin{problem}
  \label{sec:main-result:-an}
  Given $\lambda\in[0,\frac{2\pi}{n})$, is it possible to describe the zero
  configurations of polynomials in $\overline{\mathcal{D}}_{n}(\lambda)$?
\end{problem}

In the next section we will present several further properties of the classes
$\overline{\mathcal{D}}_{n}(\lambda)$, hoping that this will be of help for
finding an answer to Problem \ref{sec:main-result:-an}.

\subsection{Further properties and characterizations of the classes
  $\overline{\mathcal{D}}_{n}(\lambda)$}
\label{sec:furth-prop-char}

By definition, a polynomial $F\in\mathcal{D}_{n}(\lambda)$ either has all its
zeros on $\T$ (in this case $F$ belongs to $\mathcal{T}_{n}(\lambda)$), or
all in $\D$ (in this case $F$ satisfies (\ref{eq:22})). Even though this does
not follow directly from their definition, the classes
$\overline{\mathcal{D}}_{n}(\lambda)$ can be partitioned in the same way.

\begin{theorem}
  \label{sec:an-extens-suffr-2}
  If $F\in\overline{\mathcal{D}}_{n}(\lambda)$, $\lambda\in(0,\frac{2\pi}{n}]$,
  has one zero on $\T$, then $F\in\overline{\mathcal{T}}_{n}(\lambda)$ (and thus
  has all zeros on $\T$).
\end{theorem}

Note, however, that $\overline{\mathcal{T}}_{n}(\lambda)$ is contained in the
closure of $\mathcal{D}_{n}(\lambda)\setminus \mathcal{T}_{n}(\lambda)$,
since by (\ref{eq:22}), Lemma \ref{sec:main-result:-an-1}, and
the maximum principle, we have $F(rz)\in\mathcal{D}_{n}(\lambda)\setminus
\mathcal{T}_{n}(\lambda)$ for every $F\in\overline{\mathcal{T}}_{n}(\lambda)$
and all $r>1$.

The \emph{$n$-inverse} of a polynomial $F$ of degree $\leq n$ is defined as
$F^{*n}(z):=z^{n}\overline{F(1/\overline{z})}$.  It is a well known property
of finite Blaschke products that for a polynomial $F$ of degree $n$ with at
least one zero in $\D$ we have
\begin{equation}
  \label{eq:49}
  F \in \pi_{n}(\overline{\D}) \quad \mbox{if, and only if,} \quad
  F+\zeta F^{*n} \in \overline{\mathcal{T}}_{n}(0) \quad \mbox{for all} \quad
  \zeta \in \T,
\end{equation}
and
\begin{equation}
  \label{eq:50}
  F \in \pi_{n}(\D) \quad \mbox{if, and only if,} \quad
  F+\zeta F^{*n} \in \mathcal{T}_{n}(0) \quad \mbox{for all} \quad
  \zeta \in \T. 
\end{equation}
The next result therefore justifies the definitions (\ref{eq:44}) from above. It
also shows that the classes $\overline{\mathcal{D}}_{n}(\lambda)$ and
$\mathcal{D}_{n}(\lambda)$ are decreasing with respect to $\lambda$.

\begin{theorem}[First equivalent characterization of the classes
  $\overline{\mathcal{D}}_{n}(\lambda)$]
  \label{sec:an-extens-suffr-4}
  Let $\lambda\in(0,\frac{2\pi}{n})$ and suppose
  $F\in\pi_{n}(\overline{\D})$.  Then 
  \begin{equation}
    \label{eq:51}
    F\in \overline{\mathcal{D}}_{n}(\lambda)\setminus
    \overline{\mathcal{T}}_{n}(\lambda)\quad \mbox{if, and only if,} 
    \quad F+\zeta F^{*n} \in
    \overline{\mathcal{T}}_{n}(\lambda)\quad\mbox{for all}\quad \zeta\in\T,
  \end{equation}
  and 
  \begin{equation}
    \label{eq:52}
    F\in \mathcal{D}_{n}(\lambda)\setminus\mathcal{T}_{n}(\lambda) 
    \quad \mbox{if, and only if,} \quad
    F+\zeta F^{*n} \in \mathcal{T}_{n}(\lambda) \quad\mbox{for all}\quad 
    \zeta\in\T.
  \end{equation}
  Moreover, if $F\in \overline{\mathcal{D}}_{n}(\lambda)\setminus
  \overline{\mathcal{T}}_{n}(\lambda)$ and if there is a $z_{0}\in\T$ such that
  \begin{equation*}
    e^{-in\lambda/2} \frac{F_{+}(z_{0})}{F_{-}(z_{0})} = x\in \R,
  \quad \mbox{then} \quad
    \lim_{z\rightarrow z_{0}} e^{-in\lambda/2}\frac{(F+\zeta F^{*n})_{+}(z)}
    {(F+\zeta F^{*n})_{-}(z)} = x \quad \mbox{for all} \quad \zeta\in\T.
  \end{equation*}
\end{theorem}

A polynomial $F$ of degree $\leq n$ is called \emph{$n$-self-inversive} if $F
= F^{*n}$. The zeros of $n$-self-inversive polynomials lie symmetrically
around $\T$, and for every $F$, lying in the set $\mathcal{ST}_{n}$ of
polynomials of degree $\leq n$ whose zeros lie symmetrically around $\T$,
there is a uniquely determined $c_{F}\in \{e^{it}:t\in[0,\pi)\}$ such that
$c_{F}F$ is $n$-self-inversive, i.e. such that
\begin{equation}
  \label{eq:55}
  (c_{F}F)^{*n} =  c_{F} F.
\end{equation}
Of course, every $F\in\overline{\mathcal{T}}_{n}(\lambda)$ belongs to
$\mathcal{ST}_{n}$ and thus, for such $F$,
\begin{equation}
  \label{eq:53}
  F + \zeta F^{*n} = (1+\zeta c_{F}^{2})F \in\overline{\mathcal{T}}_{n}(\lambda)
\end{equation}
for all except one $\zeta\in\T$. Hence, essentially Theorem
\ref{sec:an-extens-suffr-4} also holds for the classes
$\overline{\mathcal{T}}_{n}(\lambda)$.

If $F(z)= \sum_{k=0}^{n} C_{k}^{(n)}(\lambda) a_{k} z^{k}$ is 
of degree $\leq n$ and $\lambda\in(0,\frac{2\pi}{n})$, then it follows from
(\ref{eq:3}) that
\begin{equation}
  \label{eq:41}
  \Delta_{\lambda}^{n}[F](z) := \frac{F_{+}(z) - F_{-}(z)}{2iz
  \sin \frac{n\lambda}{2}} = \sum_{k=0}^{n-1} C_{k}^{(n-1)}(\lambda)  a_{k+1} z^{k}.
\end{equation}
The operator $\Delta_{\lambda}^{n}$ can be used to characterize the classes
$\overline{\mathcal{T}}_{n}(\lambda)$ (cf. \cite{suffridge} or
\cite[Thm. 7.5.1]{sheil}).

\begin{theorem}[Suffridge's $q$-extension of the Gau\ss-Lucas theorem for
  $\overline{\mathcal{T}}_{n}(\lambda)$]
  \label{sec:class-mathc-mathc}
  Let $F\in \mathcal{ST}_{n}$ and $\lambda\in[0,\frac{2\pi}{n})$. Then $F\in
  \overline{\mathcal{T}}_{n}(\lambda)$ if, and only if, $\Delta_{\lambda}^{n}[F]
  \in \pi_{n-1}(\overline{\D})$. Furthermore, $F\in
  \mathcal{T}_{n}(\lambda)$ if, and only if, $\Delta_{\lambda}^{n}[F]
  \in \pi_{n-1}(\D)$. 
\end{theorem}

From (\ref{eq:41}) one can readily deduce that
\begin{equation*}
  \Delta_{\lambda}^{n}[F]\rightarrow \frac{F'}{n} \qquad \mbox{as} 
  \qquad \lambda\rightarrow 0.
\end{equation*}
This gives the justification for setting $\Delta_{0}^{n}[F]:=F'/n$ and explains
why the above theorem is in fact a $q$-extension of the theorem of Gau\ss-Lucas.

Theorem \ref{sec:class-mathc-mathc} does not carry over completely to the
classes $\overline{\mathcal{D}}_{n}(\lambda)$. Nevertheless, we have the
following $q$-extension of the theorem of Gau\ss-Lucas for the classes
$\overline{\mathcal{D}}_{n}(\lambda)$.

\begin{theorem}[$q$-extension of the Gau\ss-Lucas theorem for
  $\overline{\mathcal{D}}_{n}(\lambda)$]
  \label{sec:class-mathc-mathc-1}
  Let $\lambda\in[0,\frac{2\pi}{n})$ and $F\in
  \overline{\mathcal{D}}_{n}(\lambda)$. Then $\Delta_{\lambda}^{n}[F]
  \in \pi_{n-1}(\overline{\D})$. Furthermore, if $F\in
  \mathcal{D}_{n}(\lambda)$, then $\Delta_{\lambda}^{n}[F] \in \pi_{n-1}(\D)$.
\end{theorem}

As shown in \cite[Thm. 18]{lam11}, for a polynomial $F$ of degree $n$ we have
$F= P-Q$ with $P,Q\in\mathcal{ST}_{n}$ and $c_{P}\neq c_{Q}$ if, and only
if, there are $\eta$, $\zeta\in\T$ with $\eta \neq \zeta$ such that
\begin{equation}
  \label{eq:28}
  P=\frac{\eta^{2} F - F^{*n}}{\eta^{2}-\zeta^{2}} \quad \mbox{and} \quad
  Q=\frac{\zeta^{2} F - F^{*n}}{\eta^{2}-\zeta^{2}}.
\end{equation}
In fact, if at least one zero of $F$ lies in $\D$, then the
Hermite-Biehler theorem (cf. Lemma \ref{sec:self-invers-polyn-2}
below) states that $F\in\pi_{n}(\D)$ if, and only if, $P$ and $Q$
belong to $\pi_{n}(\T)$ and have strictly interspersed zeros.
Theorems \ref{sec:an-extens-suffr-2} and \ref{sec:an-extens-suffr-4}
thus imply the following.

\begin{lemma}
  \label{sec:furth-prop-char-2}
  Let $F\in\overline{\mathcal{D}}_{n}(\lambda)\setminus
  \overline{\mathcal{T}}_{n}(\lambda)$. Then for $P$ and $Q$ as defined in
  (\ref{eq:28}) (with $\eta$, $\zeta\in\T$, $\eta\neq \zeta$) we have $P$,
  $Q\in \overline{\mathcal{T}}_{n}(\lambda)$, $P\curlyvee Q$, and $F= P-Q$. If
  $F\in \mathcal{D}_{n}(\lambda) \setminus \mathcal{T}_{n}(\lambda)$, then
  additionally $P$, $Q\in \mathcal{T}_{n}(\lambda)$.
\end{lemma}

The converse of this statement does not hold, i.e. if there are $P$, $Q\in
\overline{\mathcal{T}}_{n}(\lambda)$ with $P\curlyvee Q$, then it is not
necessarily true that $F= P-Q$ belongs to $\overline{\mathcal{D}}_{n}(\lambda)$
(cf. the remarks following Theorem \ref{sec:an-extens-suffr-7}
below). We have, however, the following two characterizations of the classes
$\overline{\mathcal{D}}_{n}(\lambda)$ in terms of the decomposition $F = P-Q$.

\begin{theorem}[Second equivalent characterization of the classes
  $\overline{\mathcal{D}}_{n}(\lambda)$]
  \label{sec:an-extens-suffr-6}
  Let $\lambda\in[0,\frac{2\pi}{n})$. Suppose $P$, $Q\in \pi_{n}(\T)$ are
  such that $c_{P} \neq c_{Q}$ and set $F:=P-Q$. Then
  $F\in\overline{\mathcal{D}}_{n}(\lambda)\setminus
  \overline{\mathcal{T}}_{n}(\lambda)$
  if, and only if,
  \begin{equation*}
    R := \frac{c_{P}\Delta_{\lambda}^{n}[P]}{c_{Q}\Delta_{\lambda}^{n}[Q]}
  \end{equation*}
  maps $\C\setminus\overline{\D}$ into the upper or lower
  half-plane. Moreover,
  $F\in\mathcal{D}_{n}(\lambda)\setminus\mathcal{T}_{n}(\lambda)$ if, and
  only if, $R$ maps $\C\setminus\D$ into the open upper or lower half-plane.
\end{theorem}

\begin{theorem}[Third equivalent characterization of the classes
  $\overline{\mathcal{D}}_{n}(\lambda)$]
  \label{sec:an-extens-suffr-8}
  Let $\lambda\in(0,\frac{2\pi}{n})$. Suppose $P$, $Q\in \pi_{n}(\T)$
  are such that $c_{P} \neq c_{Q}$ and set $F:=P-Q$ and
  \begin{equation}
    \label{eq:46}
    S:=P_{+}\cdot Q_{-}-P_{-}\cdot Q_{+}, 
    \quad \mbox{and}\quad
    T:=F_{+}\cdot (F^{*n})_{-}-F_{-}\cdot (F^{*n})_{+}.
  \end{equation}
  Then the following holds:
  \begin{enumerate}
  \item\label{item:1} If $F\in\pi_{n}(\D)$ and all zeros of $S$ or $T$
    that lie on $\T$ are of even order, then $F\in
    \overline{\mathcal{D}}_{n}(\lambda)\setminus
    \overline{\mathcal{T}}_{n}(\lambda)$.
  \item\label{item:5} If $F\in \overline{\mathcal{D}}_{n}(\lambda)\setminus
    \overline{\mathcal{T}}_{n}(\lambda)$, then all zeros of $S$ and
    $T$ on $\T$ are of even order.
  \item\label{item:6} If $F\in\pi_{n}(\D)$ and if $S$ or $T$ has exactly
    $n-1$ critical points in $\overline{\D}$, then
    $F\in\mathcal{D}_{n}(\lambda)\setminus \mathcal{T}_{n}(\lambda)$.
  \item\label{item:7} If $F\in\mathcal{D}_{n}(\lambda)\setminus
    \mathcal{T}_{n}(\lambda)$, then $S$ and $T$ have exactly $n-1$
    critical points in $\D$.
  \end{enumerate}
\end{theorem}

By \cite[Thm. 7.1.3]{sheil} the polynomial $T$ in
(\ref{eq:46}) has exactly $n-1$ critical points in $\overline{\D}$
if, and only if, it does not vanish on $\T$. Hence, another way to
state Theorem \ref{sec:an-extens-suffr-8} would be that $F$ belongs to
$\mathcal{D}_{n}(\lambda)$ if, and only if, the finite Blaschke
product $B:= F/F^{*n}$ satisfies
\begin{equation*}
  B(e^{i\lambda/2} z) \neq B(e^{-i\lambda/2} z)
  \quad \mbox{for all} \quad z\in\T.
\end{equation*}
The following is thus merely a reformulation of Problem
\ref{sec:main-result:-an}.

\begin{problem}
  Let $\lambda\in[0,\frac{2\pi}{n})$. Is it possible to obtain a description
  of the zero configurations of those Blaschke products of degree $n$
  that map every arc on $\T$ of length $\lambda$ onto an arc of length less
  than $2\pi$?
\end{problem}

Observe that Theorem \ref{sec:an-extens-suffr-8} provides a feasible way to
check whether a given polynomial $F$ belongs to
$\overline{\mathcal{D}}_{n}(\lambda)$ or not. In general, however, we cannot
give many concrete examples of polynomials in
$\overline{\mathcal{D}}_{n}(\lambda)$. At the moment the only polynomials in
$\overline{\mathcal{D}}_{n}(\lambda)$ which we know concretely are (1) all
polynomials in $\overline{\mathcal{T}}_{n}(\lambda)$, (2) all polynomials
whose zeros lie in $|z|<r_{n,\lambda}$, where $r_{n,\lambda}$ is a number in
$(0,1)$ whose existence follows from the fact that
$z^{n}\in\mathcal{D}_{n}(\lambda)$ for all $\lambda\in[0,\frac{2\pi}{n})$,
and (3) all polynomials of the form $F(rz)$ where $F$ is any given polynomial
in $\overline{\mathcal{D}}_{n}(\lambda)$ and $r>1$ (this follows directly
from the definition of the classes $\overline{\mathcal{D}}_{n}(\lambda)$).
A perhaps more interesting subset of $\overline{\mathcal{D}}_{n}(\lambda)$)
is presented in the next theorem.





\begin{theorem}
  \label{sec:an-extens-suffr-7}
  Let $\lambda\in(0,\frac{2\pi}{n})$ and suppose
  $P\in\pi_{n}(\T)$. Then $F(z):=P(z) - Q_{n}(\lambda;z)$ belongs to
  $\overline{\mathcal{D}}_{n}(\lambda)\setminus
  \overline{\mathcal{T}}_{n}(\lambda)$ if, and only if, there are
  $c\in\T\setminus\{\pm 1\}$ and $a,b\in\R$, $a\neq 0$, such that
  \begin{equation}
    \label{eq:56}
    P(z) = c\left(b + a \sum_{k=1}^{n} 
      \frac{e^{i(k-n-1)\frac{\lambda}{2}}}{\sin \frac{(k-n-1)\lambda}{2}}
      \frac{1+e^{i(n+1)\frac{\lambda}{2}}z}{1+e^{i(2k-n-1)\frac{\lambda}{2}}z} \right)\, 
    Q_{n}(\lambda;z).
  \end{equation}
\end{theorem}

By definition of the classes $\mathcal{P}\overline{\mathcal{D}}_{n}(\lambda)$
and $\mathcal{PD}_{n}(\lambda)$ Theorem \ref{sec:an-extens-suffr-3}\ref{item:3}
can also be stated in the following form: if $\lambda<\mu<\frac{2\pi}{n}$ and if
$F\in \overline{\mathcal{D}}_{n}(\lambda)$ is such that $F + \zeta F^{*n}$ is
not $\lambda$-extremal for any $\zeta\in \T$, then $F*_{\lambda} Q_{n}(\mu;z)
\in \mathcal{D}_{n}(\mu)$. Theorem \ref{sec:an-extens-suffr-7} shows that there
are in fact $F\in \overline{\mathcal{D}}_{n}(\lambda)\setminus
\overline{\mathcal{T}}_{n}(\lambda)$ with the property that $F+\zeta F^{*n}$ is
$\lambda$-extremal for a $\zeta\in\T$. For, if $F = P-Q_{n}(\lambda;z)$ with $P$
as in Theorem \ref{sec:an-extens-suffr-7}, then $F\in
\overline{\mathcal{D}}_{n}(\lambda)\setminus
\overline{\mathcal{T}}_{n}(\lambda)$, and $F-c^{2} F^{*n} = (c^{2}-1)
Q_{n}(\lambda;z)$.

Since by Theorem \ref{sec:an-extens-suffr-7} the set of polynomials
$P$ for which $F=P-Q_{n}(\lambda;z)$ belongs to
$\overline{\mathcal{D}}_{n}(\lambda)\setminus
\overline{\mathcal{T}}_{n}(\lambda)$ is a three-parameter family, it
is clear that for large enough $n$ there will be a polynomial
$P\in\mathcal{T}_{n}(\lambda)$ with $P\curlyvee Q_{n}(\lambda)$ such
that $F=P-Q$ does not belong to
$\overline{\mathcal{D}}_{n}(\lambda)$. This proves that the converse
of Lemma \ref{sec:furth-prop-char-2} does not hold.

Next, let $\mathcal{R}_{1}$ denote the set of functions $f$ analytic in
$\D$ for which $f(0)=1$ and $\Re f(z)>\frac{1}{2}$ for $z\in\D$. In
\cite{suffridge} Suffridge showed that $f\in\mathcal{R}_{1}$ if, and only if,
there are sequences $(n_{k})_{k}\subset \N$, $(\lambda_{k})_{k}$, and
$(p_{k})_{k}$, with $n_{k}\geq k$, $\lambda_{k}\in(0,\frac{2\pi}{n_{k}})$,
$p_{k}\in\mathcal{PT}_{n_{k}}(\lambda_{k})$, and $p_{k}(0)=1$, for all
$k\in\N$, such that $p_{k}\rightarrow f$ uniformly on compact subsets of $\D$
as $k\rightarrow \infty$. One might hope that the limits of polynomials in
the larger class $\mathcal{PD}_{n}(\lambda)$ (in fact, limits of the
$n$-inverses of polynomials in $\mathcal{PD}_{n}(\lambda)$, since we want
convergence in $\D$) constitute a larger class of functions than
$\mathcal{R}_{1}$, but this is not the case.

\begin{theorem}
  \label{sec:an-extens-suffr-5}
  Let $f$ be analytic in $\D$ with $f(0)=1$.  If there is a strictly
  increasing sequence $(n_{k})_{k}\subset \N$, and sequences
  $(\lambda_{k})_{k}$ and $(p_{k})_{k}$ with
  $\lambda_{k}\in(0,\frac{2\pi}{n_{k}})$,
  $p_{k}\in\mathcal{PD}_{n_{k}}(\lambda_{k})$, and $p_{k}^{n_{k}*}(0)=1$, for
  all $k\in\N$, such that $p_{k}^{*n_{k}}\rightarrow f$ uniformly on compact
  subsets of $\D$ as $k\rightarrow \infty$, then $f\in\mathcal{R}_{1}$.
\end{theorem}

The proofs of Theorems \ref{sec:an-extens-suffr-3}\ref{item:4} and
\ref{sec:an-extens-suffr-5}, as well as Suffridge's approximation technique
from \cite{suffridge}, have led us to a new proof of the following
version of the Herglotz representation formula.

\begin{theorem}[Herglotz representation formula $\mbox{\cite{herglotz}}$]
  \label{sec:furth-prop-char-3}
  A function $f$ analytic in $\D$ satisfies $f(0)=1$ and $\Re f(z) > 0$ for
  all $z\in\D$ if, and only if, there is a strictly increasing sequence
  $(m_{n})_{n}\subset \N$ and positive numbers $s_{k}^{(n)}$,
  $k\in\{1,\ldots,m_{n}\}$ with $s_{1}^{(n)} + \cdots + s_{m_{n}}^{(n)} = 1$
  for all $n\in\N$ such that
  \begin{equation*}
    f(z) = \lim_{n\rightarrow \infty} \sum_{k=1}^{m_{n}} s_{k}^{(n)} \frac{1+
      e^{2\pi i k/m_{n}}z}{1-e^{2\pi i k/m_{n}}z} \qquad \mbox{uniformly on
      compact subsets of } \D. 
  \end{equation*}
\end{theorem}

We will present our new proof of the Herglotz representation formula,
together with the proofs of all other results in this paper, in Section
\ref{sec:proofthm2}.

\subsection{A continuous transition between the disk and half-plane
  cases of the Grace-Szeg\"o convolution theorem}
\label{sec:cont-conn-betw}

Since they stem from the choice $K=\overline{\D}$, Corollaries
\ref{sec:introduction-3} and \ref{sec:introduction-4} can be regarded as
'disk cases' of the Grace-Szeg\"o convolution theorem. If we consider $K$ to
be equal to a half-plane, then the following 'half-plane case' can be deduced
from Grace's theorem in the same way as Corollary \ref{sec:introduction-3}
(note that $1+az\rightarrow 1$ as $a\rightarrow 0$, which explains why in the
case of unbounded sets we can also make statements about polynomials of
degree $<n$).

\begin{corollary}
  \label{sec:introduction-5}
  Let $H$ be a closed half-plane whose boundary contains the origin and
  suppose $Q$ is of degree $\leq n$. Then $P*_{GS}Q\in\pi_{\leq n}(H)$ for
  all $P\in\pi_{\leq n}(H)$ if, and only if, $Q\in\pi_{\leq n}(\R_{0}^{-})$,
  where $\R_{0}^{-}:=\{z\in\R:z\leq 0\}$.
\end{corollary}

By considering $H$ to be equal to the upper, lower, and left
half-plane, this implies the following.

\begin{corollary}
  \label{sec:introduction-6} Let $Q$ be of degree $\leq n$. Then
  $P*_{GS}Q\in\pi_{\leq n}(\R)$ for all $P\in\pi_{\leq n}(\R)$ if, and only
  if, $Q\in\pi_{\leq n}(\R_{0}^{-})$, and $P*_{GS}Q\in\pi_{\leq
    n}(\R_{0}^{-})$ for all $P\in\pi_{\leq n}(\R_{0}^{-})$ if, and only if,
  $Q\in\pi_{\leq n}(\R_{0}^{-})$.
\end{corollary}

Disks and half-planes are obvious candidates when looking for zero regions in
the complex plane that stay invariant under the Grace-Szeg\"o convolution
since they appear explicitly in the statement of Grace's theorem. It seems,
however, that until now the question whether there is a continuous transition
between the disk and the half-plane cases has not been considered. This is
even more surprising since very recently the question of linear mappings in
$\C_{n}[z]$ which preserve $\pi_{n}(\Omega)$ for disks, half-planes, and
their boundaries, was completely solved by Borcea and Br\"and\'en
\cite{borcebraend09} (see also \cite[Thm. 1.1]{rusch82} for the linear
preservers of $\pi_{n}(\D)$).

Our interest in this question was strongly motivated by a recent series of
papers by Ruscheweyh and Salinas (and Sugawa)
\cite{ruschsal08,ruschsal09,ruschsalsug09}, in which a limit version of
Theorem \ref{sec:an-extens-suffr-1} is extended to domains other than $\D$.
More exactly, as explained in \cite{ruschsal09}, for every open disk or
half-plane $\Omega$ that contains the origin there are two unique parameters
$\tau \in \mathbb{C}\setminus \{0\}$ and $\gamma\in [0,1]$ such that $\Omega$
is the image of $\D$ under a M\"obius transformation of the form
\begin{equation*}
  w_{\tau,\gamma} (z):= \frac{\tau z}{1 + \gamma z}.
\end{equation*}
We write $\Omega_{\tau,\gamma}$ for such a circular domain and note that, for
all $\tau\in \mathbb{C}\setminus \{0\}$, $\Omega_{\tau,0}$ is a disk centered
at the origin and $\Omega_{\tau,1}$ is an open half-plane containing the
origin. For $\gamma\in[0,1)$ we also define
\begin{align*}
  I_{\gamma}:=&\; \{z\in\C:|z| + \gamma |1+z| < 1\}, \\
  O_{\gamma}:=&\; \{z\in\C:|z| - \gamma |1+z| > 1\}.
\end{align*}
$\overline{I}_{\gamma}$ and $\overline{O}_{\gamma}$ are families of sets which,
when $\gamma$ increases from $0$ to $1$, decrease from
$\overline{I}_{0}=\overline{\D}$ and $\overline{O}_{0}=\C\setminus\D$ to
\begin{equation}
  \label{eq:13}
  \overline{I}_{1}:=\bigcap_{\gamma\in[0,1)} \overline{I}_{\gamma} = [-1,0]
  \quad \mbox{and} \quad
  \overline{O}_{1}:=\bigcap_{\gamma\in[0,1)} \overline{O}_{\gamma} = (-\infty,-1],
\end{equation}
respectively. For $\gamma\in(0,1)$, $I_{\gamma}$ is the interior of the inner
loop of the limacon of Pascal, and $O_{\gamma}$ is the open exterior of the
limacon of Pascal. See \cite{ruschsal09} and the figures therein and note
that, with $L_{\gamma}$ and $\Omega_{\gamma}^{*}$ as defined there, we have
$I_{\gamma} = - L_{\gamma}$ and $O_{\gamma} =-
\C\setminus\overline{\Omega}_{\gamma}^{*}$. 

The non-circular sets $I_{\gamma}$ and $O_{\gamma}$ are
zero-domains which stay invariant under Grace-Szeg\"o convolution.

\begin{theorem}
  \label{sec:main-results}
  Let $\tau \in \mathbb{C}\setminus \{0\}$, $\gamma\in [0,1)$, and suppose $Q$
  is of degree $\leq n$. Then
  \begin{enumerate}
  \item \label{item:9} $P*_{GS}Q\in\pi_{n}(\Omega_{\tau,\gamma})$ for
    all $P\in\pi_{n}(\overline{\Omega}_{\tau,\gamma})$ if, and only
    if, $Q\in\pi_{n}(I_{\gamma})$,
  \item \label{item:15} $P*_{GS}Q\in\pi_{n}(\Omega_{\tau,\gamma})$ for
    all $P\in\pi_{n}(\Omega_{\tau,\gamma})$ if, and only
    if, $Q\in\pi_{n}(\overline{I}_{\gamma})$,
  \item \label{item:10} $P*_{GS}Q\in\pi_{\leq
      n}(\C\setminus\overline{\Omega}_{\tau,\gamma})$ for all
    $P\in\pi_{\leq n}(\C\setminus\Omega_{\tau,\gamma})$ if, and only if,
    $Q\in\pi_{\leq n}(O_{\gamma})$,
  \item \label{item:16} $P*_{GS}Q\in\pi_{\leq
      n}(\C\setminus\overline{\Omega}_{\tau,\gamma})$ for all
    $P\in\pi_{\leq n}(\C\setminus\overline{\Omega}_{\tau,\gamma})$ if, and
    only if, $Q\in\pi_{\leq n}(\overline{O}_{\gamma})$,
  \item \label{item:11} $P*_{GS}Q\in\pi_{n}(I_{\gamma})$ for all
    $P\in\pi_{n}(\overline{I}_{\gamma})$ if, and only if,
    $Q\in\pi_{n}(I_{\gamma})$, and
  \item\label{item:12} $P*_{GS}Q\in\pi_{\leq n}(O_{\gamma})$ for all
    $P\in\pi_{\leq n}(\overline{O}_{\gamma})$ if, and only if,
    $Q\in\pi_{\leq n}(O_{\gamma})$.
  \end{enumerate}
\end{theorem}

Note how in the case $\gamma=0$ this theorem implies the disk-cases of
Grace's theorem, while for $\gamma\rightarrow 1$ we obtain certain half-plane
cases of Grace's theorem (most notably Corollary \ref{sec:introduction-6}).

In \cite{ruschsal08,ruschsal09,ruschsalsug09} convolution
invariance results concerning certain classes of functions which are analytic
on the sets $\Omega_{\tau}$ and $I_{\gamma}$ are obtained by resorting to a
limiting case of Suffridge's theorem (i.e. Theorem
\ref{sec:an-extens-suffr-1}) which deals with the convolution of starlike and
convex univalent mappings on $\D$. It is therefore natural to ask whether
there is some kind of extension of Theorem \ref{sec:an-extens-suffr-1} from
which the results in \cite{ruschsal08,ruschsal09,ruschsalsug09} can be
obtained as limiting cases.  

\begin{problem}
  Find Suffridge-like extensions of the statements in Theorem
  \ref{sec:main-results}.
\end{problem}

As noted above, Borcea and Br{\"a}nd{\'e}n \cite{borcebraend09} found a
complete characterization of all linear operators on the space of complex
polynomials which preserve the sets $\pi_{n}(\Omega)$ and
$\pi_{n}(\partial\Omega)$ for disks or half-planes $\Omega$. Theorem
\ref{sec:main-results} naturally leads to the question for the linear
preservers of $\pi_{n}(\overline{O}_{\gamma})$ and
$\pi_{n}(\overline{I}_{\gamma})$ for $\gamma\in(0,1]$. Because of
(\ref{eq:13}) this includes the problem, posed in \cite{borcebraend09}, to
classify all linear preservers of $\pi_{n}(\Omega)$ when $\Omega$ is a ray or
a finite interval.

Theorem \ref{sec:main-results} is essentially only a corollary of Borcea's
and Br\"and\'en's results from \cite{borcebraend09}. Nevertheless in Section
\ref{sec:proofthm2} we will give a short self-contained proof which uses only
the original Grace-Szeg\"o convolution theorem.

\section{Preliminaries}

Before we can proceed to the proofs of the main results, we need to mention
certain facts regarding $n$-inverses and $n$-self-inversive
polynomials (many of them are fairly obvious, some of them are explained in
more detail in \cite[Sec. 7.1]{sheil}). 

Recall that above, for a polynomial $P(z)=\sum_{k=0}^{n} a_{k} z^{k}$
of degree $\leq n$, we defined
\begin{equation*}
  P^{*n}(z) := z^{n} \overline{P\left(\frac{1}{\overline{z}}\right)} = 
  \sum_{k=0}^{n} \overline{a}_{n-k} z^{k}
\end{equation*}
and called $P^{*n}$ the \emph{$n$-inverse} of $P$. The zeros of $P^{*n}$
are the zeros of $P$ reflected with respect to $\T$ (if we consider a polynomial
$P$ of degree $m<n$ as a polynomial with a zero of order $n-m$ at
$\infty$), and the mapping $P\mapsto P^{*n}$ has
the following properties (in order to verify the last one, note that, by
(\ref{eq:3}), one has $C_{n-k}^{(n)}(\lambda) = C_{k}^{(n)}(\lambda)\in\R$ for
$k=0,\ldots,n$):
\begin{equation}
  \label{eq:31}
  \begin{split}
    (aP+bQ)^{*n} =  \overline{a} P^{*n} +\overline{b} Q^{*n} \quad&
    \mbox{for all} \quad a,b\in\C, \; P,Q\in \pi_{\leq n}(\C),\\
    (P(cz))^{*n} = \;\overline{c}^{n} P^{*n}(cz) \quad &\mbox{for all}
    \quad c \in \T, \; P \in\pi_{\leq n}(\C), \\
    (P*_{\lambda} Q)^{*n} = \; P^{*n}*_{\lambda} Q^{*n} \quad &\mbox{for all}
    \quad \lambda \in [0,2\pi/n), \; P,Q \in\pi_{\leq n}(\C).
  \end{split}
\end{equation}

A polynomial $P$ of degree $\leq n$ is called \emph{$n$-self-inversive} if $P =
P^{*n}$, and every such polynomial belongs to the set $\mathcal{ST}_{n}$ of
polynomials of degree $\leq n$ whose zeros lie symmetrically around
$\T$. Conversely, for every $P\in\mathcal{ST}_{n}$ there is a uniquely
determined $c_{P}\in \{e^{it}:t\in[0,\pi)\}$ such that $c_{P}P$ belongs
to $\mathcal{SI}_{n}$. It follows that if $P\in\mathcal{ST}_{n}$, then
\begin{equation}
  \label{eq:48}
  e^{-int/2}c_{P}P(e^{it})\in\R \quad\mbox{for all}\quad t\in\R,
\end{equation}
and thus that
\begin{equation*}
  \Re \frac{e^{it} P'(e^{it})}{P(e^{it})} = \frac{n}{2}
  \quad\mbox{for all}\quad t\in\R \quad \mbox{with} \quad P(e^{it})\neq 0.
\end{equation*}
In particular,
\begin{equation}
  \label{eq:47}
  \mbox{if}\quad P\in\mathcal{ST}_{n} \quad \mbox{then every critical point 
    of $P$ on $\T$ is a multiple zero of $P$.}
\end{equation}

If $P$ and $Q$ belong to $\pi_{n}(\T)$, then we say that $P$ and $Q$ have
\emph{interspersed zeros} and write $P\closedcurlyvee Q$ if the zeros of $P$ and
$Q$ alternate on $\T$. If $P\closedcurlyvee Q$ and $P$ and $Q$ are co-prime,
then we write $P\curlyvee Q$ and say that $P$ and $Q$ have \emph{strictly
  interspersed zeros}.

The following connection between interspersion and zero separation is proven in
\cite[Lem. 21]{lam11} (cf. (\ref{eq:9}) for the definition of $P_{+}$ and
$P_{-}$).

\begin{lemma}
  \label{sec:defin-prel-1}
  Suppose all zeros of $P$ lie on $\T$ and let
  $\lambda\in[0,\frac{2\pi}{n}]$. Then
  $P\in\overline{\mathcal{T}}_{n}(\lambda)$ if, and only if,
  $P_{+}\closedcurlyvee P_{-}$. Furthermore,
  $P\in\mathcal{T}_{n}(\lambda)$ if, and only if,
  $P_{+}\curlyvee P_{-}$.
\end{lemma}

Two characterizations of interspersion will be useful in the following.

\begin{lemma}
  \label{sec:self-invers-polyn-2}
  Suppose $P$, $Q\in \mathcal{ST}_{n}$ are such that $P/Q\nequiv
  \operatorname{const}$ and set $F:=P-Q$.
  \begin{enumerate}
  \item \label{item:17} (Hermite-Biehler) If $c_{P} \neq c_{Q}$, then
    $P\closedcurlyvee Q$ if, and only if, $F$ or
    $F^{*n}$ belongs to $\pi_{n}(\overline{\D})$. $P\curlyvee Q$ holds if, and
    only if, either $F$ or $F^{*n}$ belongs to $\pi_{n}(\D)$.
  \item \label{item:18} (Hermite-Kakeya) If $c_{P} = c_{Q}$, then $P$ and $Q$
    have interspersed zeros if, and only if,
    \begin{equation}
      \label{eq:18}
      P-xQ \in \overline{\mathcal{T}}_{n}(0) \quad \mbox{for all} \quad x\in \R.
    \end{equation}
    $P$ and $Q$ have strictly interspersed zeros if, and
    only if, $Q\in\mathcal{T}_{n}(0)$ and
    \begin{equation}
      \label{eq:19}
      P-xQ \in \mathcal{T}_{n}(0) \quad\mbox{for all} \quad x\in \R.
    \end{equation}
  \end{enumerate}
\end{lemma}

\begin{proof}
  Statement (\ref{item:17}) and the 'only if' direction of (\ref{item:18})
  were shown in \cite[Thm. 18]{lam11}.

  If (\ref{eq:18}) holds, then $R(z):=P(z)/Q(z)$ is not constant (by
  hypothesis) and takes real values if, and only if, $z\in\T$. $R$ therefore
  maps $\D$ either onto the upper or lower half-plane. Hence, all zeros of
  $P-iQ$ lie either in $\overline{\D}$ or in $\C\setminus \D$. Since $c_{iQ}
  = \pm ic_{Q}\neq c_{P}$, it follows from (\ref{item:17}) that $P$ and $iQ$
  have interspersed zeros.

  If $Q\in \mathcal{T}_{n}(0)$ and $P\closedcurlyvee Q$
  with a common zero $w=e^{it_{0}}$, then
  \begin{equation}
    \label{eq:7}
    P-xQ = (z-w)(\hat{P}-x\hat{Q})
  \end{equation}
  where $\hat{P}:=P/(z-w)$, $\hat{Q}:=Q/(z-w)$ and $\hat{Q}(w)\neq 0$.
  Since for 
  \begin{equation*}
    x_{0} = \frac{\hat{P}(w)}{\hat{Q}(w)} = \lim_{t\rightarrow t_{0}}
    \frac{e^{-int/2}c_{P}P(e^{it})}{e^{-int/2}c_{Q}Q(e^{it})}
  \end{equation*}
  we have $x_{0}\in\R$, by (\ref{eq:48}), and $\hat{P}(w)-x_{0}\hat{Q}(w)=0$,
  it follows from (\ref{eq:7}) that $P-x_{0}Q$ has a double zero at $w$.
\end{proof}

Observe that the Hermite-Biehler theorem and Lemma \ref{sec:defin-prel-1}
directly imply Theorem \ref{sec:an-extens-suffr-4}.

\section{Proofs}
\label{sec:proofthm2}

\begin{proof}[Proof of Theorem \ref{sec:an-extens-suffr-4}]
  The proofs of (\ref{eq:51}) and (\ref{eq:52}) are very similar, and therefore
  we will only verify (\ref{eq:51}).
  
  Note first that $F$ must have at least one zero in $\D$: if
  $F\in\overline{\mathcal{D}}_{n}(\lambda)\setminus
  \overline{\mathcal{T}}_{n}(\lambda)$, then this is clear; if $F+\zeta F^{*n}
  \in\overline{\mathcal{T}}_{n}(\lambda)$ for all $\zeta\in\T$, then this
  follows from (\ref{eq:53}). Consequently, (\ref{eq:8}) is equivalent to the
  fact that, for all $x\in\R$
  \begin{equation}
    \label{eq:54}
    G_{x}:=e^{-in\lambda/4}F_{+} - x e^{in\lambda/4}F_{-}
    \in \pi_{n}(\overline{\D}) \quad\mbox{with at least one zero in }\D,
  \end{equation}
  i.e. to the fact that 
  \begin{equation}
    \label{eq:26}
    G_{x}/G_{x}^{*n} \quad \mbox{is a Blaschke product of degree $m=1,\ldots,n$
  for all $x\in\R$.}
  \end{equation}
  Since $F\in\pi_{n}(\overline{\D})$ with at least one zero in $\D$, $F/F^{*n}$
  is a Blaschke product of positive degree, and thus (use (\ref{eq:31}) to
  calculate $G_{x}^{*n}(0)$)
  \begin{equation}
    \label{eq:23}
    |(G_{x}/G_{x}^{*n})(0)| = |(F/F^{*n})(0)|<1.
  \end{equation}
  Hence, (\ref{eq:26}) holds if, and only if, ((\ref{eq:23}) is needed for
  the 'if'-direction)
  \begin{equation}
    \label{eq:24}
    G_{x} + \zeta G_{x}^{*n} \in \overline{\mathcal{T}}_{n}(0) 
    \quad\mbox{for all}\quad  x\in\R,\;\zeta\in\T.
  \end{equation}
  Using (\ref{eq:31}), we obtain
  \begin{equation*}
    G_{x} + \zeta G_{x}^{*n} = e^{-in\lambda/4}\left(F+\zeta
      F^{*n}\right)_{+} - x
    e^{in\lambda/4}\left(F+\zeta 
      F^{*n}\right)_{-}.
  \end{equation*}
  It thus follows from the
  Hermite-Kakeya theorem that (\ref{eq:24}) is equivalent to
  \begin{equation*}
    \left(F+\zeta F^{*n}\right)_{+} \closedcurlyvee \left(F+\zeta
      F^{*n}\right)_{-} \quad\mbox{for all}\quad
    \zeta\in\T.
  \end{equation*}
  Because of Lemma \ref{sec:defin-prel-1} this is true if, and only if, $F +
  \zeta F^{*n} \in \overline{\mathcal{T}}_{n}(\lambda)$ for all $\zeta\in\T$.

  If $F\in\overline{\mathcal{D}}_{n}(\lambda)\setminus
  \overline{\mathcal{T}}_{n}(\lambda)$ and $e^{-in\lambda/2}
  F_{+}(z_{0})/F_{-}(z_{0}) = x_{0} \in \R$ for a $z_{0}\in\T$, then we have
  $G_{x_{0}}(z_{0}) = G_{x_{0}}^{*n}(z_{0}) = 0$, where $G_{x}$ is defined as
  in (\ref{eq:54}). From our considerations above it then follows that
  \begin{equation*}
    G_{x_{0}}(z_{0}) + \zeta G_{x_{0}}^{*n}(z_{0}) = 
    e^{-in\lambda/4}\left(F+\zeta F^{*n}\right)_{+} (z_{0})- x_{0}
    e^{in\lambda/4}\left(F+\zeta F^{*n}\right)_{-} (z_{0}) = 0
  \end{equation*}
  for all $\zeta\in\T$. Hence,
  \begin{equation*}
    e^{-in\lambda/2} \frac{\left(F+\zeta F^{*n}\right)_{+} (z_{0})}
    {\left(F+\zeta F^{*n}\right)_{-} (z_{0})}=x_{0},
  \end{equation*}
  for every $\zeta\in\T$ except for those finitely many $\zeta\in\T$ for which
  $\left(F+\zeta F^{*n}\right)_{-} (z_{0}) = 0$. However, if $\zeta_{0}$ is one
  of these $\zeta$, then
  \begin{equation*}
    \lim_{z\rightarrow z_{0}} 
    e^{-in\lambda/2} \frac{\left(F+\zeta_{0} F^{*n}\right)_{+} (z)}
    {\left(F+\zeta_{0} F^{*n}\right)_{-} (z)}
    =
    \lim_{\zeta\rightarrow\zeta_{0}} 
    e^{-in\lambda/2} \frac{\left(F+\zeta F^{*n}\right)_{+} (z_{0})}
    {\left(F+\zeta F^{*n}\right)_{-} (z_{0})} = 
    x_{0}.
  \end{equation*}
\end{proof}

\begin{proof}[Proof of Theorem \ref{sec:an-extens-suffr-2}]
  For $\lambda = \frac{2\pi}{n}$ the assertion follows directly from the
  definition of the class $\overline{\mathcal{D}}_{n}(\frac{2\pi}{n})$.  If
  $\lambda\in(0,\frac{2\pi}{n})$ and
  $F\in\overline{\mathcal{D}}_{n}(\lambda)\setminus
  \overline{\mathcal{T}}_{n}(\lambda)$, then $F+\zeta F^{*n} \in
  \overline{\mathcal{T}}_{n}(\lambda) \subset \mathcal{T}_{n}(0)$ for all
  $\zeta\in \T$ by Theorem \ref{sec:an-extens-suffr-4}, and thus
  $F\in\pi_{n}(\D)$ by (\ref{eq:50}).
\end{proof}

\begin{proof}[Proof of Theorem \ref{sec:an-extens-suffr-3}\ref{item:2}]
  The case $\lambda=0$ is Corollary \ref{sec:introduction-3} and thus it
  suffices to consider only the case $\lambda\in(0,\frac{2\pi}{n})$.

  If $F*_{\lambda}G \in \mathcal{D}_{n}(\lambda)$ for all
  $F\in\overline{\mathcal{D}}_{n}(\lambda)$, then choosing $F =
  Q_{n}(\lambda;z)$ yields $G \in \mathcal{D}_{n}(\lambda)$.

  In order to prove the other direction, we can suppose that we do not have
  simultaneously $F\in \overline{\mathcal{T}}_{n}(\lambda)$ and $G\in
  \mathcal{T}_{n}(\lambda)$, since in this case the assertion follows directly
  from Suffridge's convolution theorem.

  Suppose now that $F\in \overline{\mathcal{D}}_{n}(\lambda)\setminus
  \overline{\mathcal{T}}_{n}(\lambda)$ and $G\in \mathcal{D}_{n}(\lambda)$.
  Then, by Theorem \ref{sec:an-extens-suffr-4}, $P:=G+\zeta e^{it}G^{*n}\in
  \mathcal{T}_{n}(\lambda)$ for all except at most one $t\in[0,2\pi)$ (such
  an exceptional $t$ exists if, and only if, $G\in \mathcal{T}_{n}(\lambda)$
  and for this $t$ we have $P\equiv 0$) and $F + \eta F^{*n}\in
  \overline{\mathcal{T}}_{n}(\lambda)$ for all $\eta\in\T$. Hence, by
  Suffridge's convolution theorem,
  \begin{equation*}
    P*_{\lambda}(F + \eta F^{*n}) = P*_{\lambda}F +
    \eta\overline{c}_{P}^{2} (P*_{\lambda}F)^{*n}\in \mathcal{T}_{n}(\lambda)
  \end{equation*}
  for all $\eta\in\T$ and thus 
  \begin{equation}
    \label{eq:15}
    P*_{\lambda}F = (G+\zeta e^{it}G^{*n})*_{\lambda}F
  \in\mathcal{D}_{n}(\lambda)
  \end{equation}
  by Theorem \ref{sec:an-extens-suffr-4}. If $G\in\mathcal{T}_{n}(\lambda)$,
  then $G+\zeta e^{it}G^{*n} = (1+\zeta e^{it}c_{G}^{2})G$, and hence in the
  case $F\in \overline{\mathcal{D}}_{n}(\lambda)\setminus
  \overline{\mathcal{T}}_{n}(\lambda)$ and $G\in \mathcal{T}_{n}(\lambda)$ the
  assertion follows directly from (\ref{eq:15}). The case $F\in
  \overline{\mathcal{T}}_{n}(\lambda)$ and
  $G\in\mathcal{D}_{n}(\lambda)\setminus \mathcal{T}_{n}(\lambda)$ can be proven
  in a similar way, and therefore it only remains to verify the assertion when
  $F\in \overline{\mathcal{D}}_{n}(\lambda)\setminus
  \overline{\mathcal{T}}_{n}(\lambda)$ and
  $G\in\mathcal{D}_{n}(\lambda)\setminus \mathcal{T}_{n}(\lambda)$.

  Choose such $F$ and $G$ and note that, because of Lemma
  \ref{sec:an-extens-suffr-2}, all zeros of $F$ and $G$ must lie in $\D$.
  Moreover, $F$ and $G$ must satisfy (\ref{eq:15}) and thus, by definition,
  for all $t,x\in\R$ all solutions of
  \begin{equation*}
    e^{-in\lambda/2}\frac{((G + \zeta e^{it} G^{*n})*_{\lambda} F)_{+}} 
    {((G + \zeta e^{it} G^{*n})*_{\lambda} F)_{-}} = x
  \end{equation*}
  lie in $\D$. Hence, for all $t,x\in\R$,
  \begin{equation}
    \label{eq:32}
    B_{t}(z):=(e^{-in\lambda/4} F_{+}-xe^{in\lambda/4}F_{-})*_{\lambda}
    (G + \zeta e^{it} G^{*n})\in\pi_{n}(\D).
  \end{equation}
  Consequently, $B_{t}/B_{t}^{*n}$ is a Blaschke product of degree $n$ which
  implies that
  \begin{equation}
    \label{eq:29}
    A_{s,t}:=e^{is} B_{t}- \zeta e^{it} B_{t}^{*n} \in
    \mathcal{T}_{n}(0) \quad \mbox{for all} \quad s,t\in\R,
  \end{equation}
  and that the zeros of $A_{s,t}(z)$ are continuous and arg-decreasing with
  respect to $s$. On the other hand, using the relations (\ref{eq:9}),
  (\ref{eq:31}), and the fact that the operation $*_{\lambda}$
  is associative, we find
  \begin{equation}
    \label{eq:25}
    \begin{split}
      A_{s,t} = &\; e^{is} (e^{-in\lambda/4}
      F_{+}-xe^{in\lambda/4}F_{-})*_{\lambda}(G + \zeta e^{it}G^{*n})-\\
      &\;-\zeta e^{it} \left[(e^{-in\lambda/4}
        F_{+}-xe^{in\lambda/4}F_{-})*_{\lambda}(G + \zeta
        e^{it}G^{*n})\right]^{*n}\\
      =&\; e^{is}e^{-in\lambda/4} (F*_{\lambda}G)_{+} + \zeta e^{i(s+t)}
      e^{-in\lambda/4} (F*_{\lambda}G^{*n})_{+}-\\
      &\;-xe^{is}e^{in\lambda/4}(F*_{\lambda}G)_{-} - x\zeta e^{i(s+t)}e^{in\lambda/4}(F*_{\lambda}G^{*n})_{-}\\
      &\;-\zeta e^{it}\left[e^{-in\lambda/4} (F*_{\lambda}G)_{+} +\zeta
        e^{it} e^{-in\lambda/4} (F*_{\lambda}G^{*n})_{+} - \right.\\
      &\qquad \left.-x e^{in\lambda/4} (F*_{\lambda} G)_{-} -x\zeta
        e^{it} e^{in\lambda/4} (F*_{\lambda}G^{*n})_{-}\right]^{*n} \\
      =&\; e^{is}e^{-in\lambda/4} (F*_{\lambda}G)_{+} + \zeta e^{i(s+t)}
      e^{-in\lambda/4} (F*_{\lambda}G^{*n})_{+}-\\
      &\;-xe^{is}e^{in\lambda/4}(F*_{\lambda}G)_{-} - x\zeta e^{i(s+t)}e^{in\lambda/4}(F*_{\lambda}G^{*n})_{-}\\
      &\;-\zeta e^{it}e^{-in\lambda/4} (F^{*n}*_{\lambda}G^{*n})_{+} - e^{-in\lambda/4} (F^{*n}*_{\lambda}G)_{+} + \\
      &+ x \zeta e^{it} e^{in\lambda/4} (F^{*n}*_{\lambda} G^{*n})_{-} +x
      e^{in\lambda/4} (F^{*n}*_{\lambda}G)_{-} \quad,
    \end{split}
  \end{equation}
  which shows that
  \begin{equation}
    \label{eq:33}
    \begin{split}
      \overline{\zeta}e^{-i(s+t)}A_{s,t} = &\;
      \overline{\zeta}e^{-it}e^{-in\lambda/4} F*_{\lambda}G_{+} +
      e^{-in\lambda/4} F*_{\lambda}(G^{*n})_{+}-\\
      &\;-x\overline{\zeta}e^{-it}e^{in\lambda/4}F*_{\lambda}G_{-} - x e^{in\lambda/4}F*_{\lambda}(G^{*n})_{-}\\
      &\;- e^{-is}e^{-in\lambda/4} F^{*n}*_{\lambda}(G^{*n})_{+} -
      e^{-i(s+t)}\overline{\zeta}e^{-in\lambda/4} F^{*n}*_{\lambda}G_{+} + \\
      &+ x e^{-is} e^{in\lambda/4} F^{*n}*_{\lambda} (G^{*n})_{-}
      +xe^{-i(s+t)}\overline{\zeta}
      e^{in\lambda/4} F^{*n}*_{\lambda}G_{-} \\
      = &\; \overline{\zeta} e^{-it}
      (F -e^{-is} F^{*n})*_{\lambda}(e^{-in\lambda/4}G_{+} -xe^{in\lambda/4}G_{-})+ \\
      &\; + (F-e^{-is} F^{*n})*_{\lambda}(e^{-in\lambda/4} (G^{*n})_{+}-
      x e^{in\lambda/4} (G^{*n})_{-})\\
      = &\; \overline{\zeta} e^{-it}
      (F -e^{-is} F^{*n})*_{\lambda}(e^{-in\lambda/4}G_{+} -xe^{in\lambda/4}G_{-})- \\
      &\; -e^{-is}\left[(F-e^{-is} F^{*n})*_{\lambda}(e^{-in\lambda/4} G_{+}-
        x e^{in\lambda/4} G_{-})\right]^{*n}\\
      = &\; \overline{\zeta} e^{-it} C_{s}-e^{-is} C_{s}^{*n},
    \end{split}
  \end{equation}
  where 
  \begin{equation*}
    C_{s} := (F -e^{-is} F^{*n})*_{\lambda}(e^{-in\lambda/4}G_{+} -xe^{in\lambda/4}G_{-}).
  \end{equation*}
  Exchanging the roles of $F$ and $G$ in the arguments that were used to deduce
  (\ref{eq:32}) shows that $C_{s} \in\pi_{n}(\D)$ for all
  $s,x\in\R$. Consequently, (\ref{eq:33}) implies that the zeros of $A_{s,t}(z)$
  are continuous and arg-increasing with respect to $t$.

  We have thus shown that all zeros of $A_{0,0}$ lie on $\T$ and are simple
  and that, for $s$ increasing from $0$, the zeros of $A_{s,0}$ arg-decrease,
  while those of $A_{0,s}$ arg-increase. It follows that for $s>0$, but $s$
  close to $0$, we have $A_{s,0}\curlyvee A_{0,s}$. Therefore, by the
  Hermite-Kakeya theorem,
  \begin{equation}
    \label{eq:30}
    A_{s,0}-A_{0,s} \in \mathcal{T}_{n}(0),
  \end{equation}
  since it follows readily from (\ref{eq:29}) that $c_{A_{s,0}}=c_{A_{0,s}}$.
  Using (\ref{eq:25}), we find
  \begin{align*}
    \frac{A_{s,0}-A_{0,s}}{e^{is}-1} = &\; e^{-in\lambda/4}
      (F*_{\lambda}G + \zeta (F*_{\lambda} G)^{*n})_{+}-
     \\
    &\;-x e^{in\lambda/4}
    (F*_{\lambda}G + \zeta (F*_{\lambda}
      G)^{*n})_{-}  \in \mathcal{T}_{n}(0)
  \end{align*}
  for all $x\in\R$. The Hermite-Kakeya theorem therefore implies
  \begin{equation*}
    (F*_{\lambda}G + \zeta (F*_{\lambda} G)^{*n})_{+}
    \curlyvee (F*_{\lambda}G + \zeta (F*_{\lambda} G)^{*n})_{-}.
  \end{equation*}
  Hence, by Lemma \ref{sec:defin-prel-1}, $F*_{\lambda}G + \zeta
  (F*_{\lambda} G)^{*n} \in \mathcal{T}_{n}(\lambda)$ for all
  $\zeta\in\T$, and thus $F*_{\lambda} G \in
  \mathcal{D}_{n}(\lambda)$ by Theorem \ref{sec:an-extens-suffr-4}.
\end{proof}

\begin{proof}[Proof of Theorem \ref{sec:an-extens-suffr-3}\ref{item:3}]
  Note that Theorem \ref{sec:an-extens-suffr-1}\ref{item:20} is equivalent to
  the statement that for $\lambda\in[0,\frac{2\pi}{n})$ and
  $\mu\in(\lambda,\frac{2\pi}{n})$ one has
  \begin{equation*}
    P*_{\lambda} Q_{n}(\mu;z) \in \mathcal{T}_{n}(\mu) \quad\mbox{for
      all}\quad P\in\overline{\mathcal{T}}_{n}(\lambda) \quad\mbox{which are
      not $\lambda$-extremal}.
  \end{equation*}
  Hence, in order to prove the assertion, it only remains to show that
  \begin{equation}
    \label{eq:27}
    F*_{\lambda} Q_{n}(\mu;z) \in \mathcal{D}_{n}(\mu) 
  \end{equation}
  for all $F\in\overline{\mathcal{D}}_{n}(\lambda)\setminus
  \overline{\mathcal{T}}_{n}(\lambda)$ for which $F + \zeta F^{*n}$ is not
  $\lambda$-extremal for any $\zeta\in\T$.

  Now, for such $F$ it follows from Theorem \ref{sec:an-extens-suffr-4} that
  $F+\zeta F^{*n} \in \overline{\mathcal{T}}_{n}(\lambda)$ for all $\zeta\in
  \T$. Since $F+\zeta F^{*n}$ is not $\lambda$-extremal for any
  $\zeta\in\T$, Theorem \ref{sec:an-extens-suffr-1}\ref{item:20} yields,
  \begin{equation*}
    F*_{\lambda}Q_{n}(\mu;z) + \zeta (F*_{\lambda}Q_{n}(\mu;z))^{*n} =
    (F+\zeta F^{*n})*_{\lambda}Q_{n}(\mu;z) \in \mathcal{T}_{n}(\mu) 
  \end{equation*}
  for all $\zeta\in\T$, which implies (\ref{eq:27}) by Theorem
  \ref{sec:an-extens-suffr-4}.
\end{proof}

\begin{proof}[Proof of Theorem \ref{sec:an-extens-suffr-3}\ref{item:4}]
  Suppose first that there is a $\lambda\in(0,\frac{2\pi}{n})$ such that
  $f(z) = \sum_{k=0}^{n}a_{k}z^{k}\in\mathcal{PD}_{n}(\lambda)\setminus
  \mathcal{PT}_{n}(\lambda)$ and set
  \begin{equation}
    \label{eq:14}
    F_{\mu}(z):= \sum_{k=0}^{n} C_{k}^{(n)}(\mu) a_{k}z^{k} \quad \mbox{ for }
    \quad    \mu\in[\lambda,\frac{2\pi}{n}).
  \end{equation}
  Then $F_{\mu} \in \mathcal{D}_{n}(\mu)\setminus \mathcal{T}_{n}(\mu)$ for all
  $\mu\in[\lambda,\frac{2\pi}{n})$, by Theorem
  \ref{sec:an-extens-suffr-3}\ref{item:3}. Hence, by definition of the
  classes $\mathcal{D}_{n}(\mu)$,
  \begin{equation*}
    \Im
    e^{-in\mu/2}\left(\frac{F_{\mu}(e^{i\mu/2}z)}{F_{\mu}(e^{-i\mu/2}z)}-1\right)
    > \sin\frac{n\mu}{2}
    \quad \mbox{for} \quad z\in\C\setminus\D, \,
    \mu\in[\lambda,\textstyle{\frac{2\pi}{n}}). 
  \end{equation*}
  This is equivalent to
  \begin{equation}
    \label{eq:21}
    \Re
    e^{-in\mu/2}\frac{z \Delta_{\mu}^{n}[F_{\mu}](z)}{F_{\mu}(e^{-i\mu/2}z)}
    > \frac{1}{2}
    \quad \mbox{for} \quad z\in\C\setminus\D, \,
    \mu\in[\lambda,\textstyle{\frac{2\pi}{n}}). 
  \end{equation}
  Now,
  \begin{equation}
    \label{eq:34}
    z\Delta_{\mu}^{n}[F_{\mu}] = \sum_{k=1}^{n} C_{k-1}^{(n-1)}(\mu) a_{k}z^{k},
  \end{equation}
  and it follows from the definition of $Q_{n}(\lambda;z)$ that
  \begin{equation*}
    Q_{n-1}({\textstyle{\frac{2\pi}{n}}};z)=\frac{1-z^{n}}{1-z} =
    \sum_{k=0}^{n-1} z^{k} 
    \qquad \mbox{and}\qquad 
    Q_{n}({\textstyle{\frac{2\pi}{n}}};z) = 1 + z^{n}.
  \end{equation*}
  Hence,
  \begin{equation*}
    C_{k}^{(n-1)}({\textstyle{\frac{2\pi}{n}}}) = 1 \quad \mbox{and}\quad 
    C_{k}^{(n)}({\textstyle{\frac{2\pi}{n}}}) = 0 \quad \mbox{for all} \quad
    k = 1,\ldots,n-1,
  \end{equation*}
  and therefore it follows from (\ref{eq:14}) and (\ref{eq:34}) that
  \begin{equation*}
    z\Delta_{\mu}^{n}[F_{\mu}] \rightarrow f(z)-a_{0} \quad \mbox{and} \quad 
    F_{\mu}(e^{-i\mu/2}z) \rightarrow -a_{n}z^{n} + a_{0} \quad \mbox{as} \quad
    \mu\rightarrow\frac{2\pi}{n}.
  \end{equation*}
  Because of (\ref{eq:21}), this implies
  \begin{equation}
   \label{eq:20}
    \Re \frac{f(z)-a_{0}}{a_{n} z^{n} - a_{0}}
    > \frac{1}{2}
    \quad \mbox{for} \quad z\in\C\setminus\overline{\D}.
  \end{equation}
  In order to show that (\ref{eq:20}) must also hold for $z\in\T$, note that
  $\mathcal{PD}_{n}(\lambda)\setminus\mathcal{PT}_{n}(\lambda)$ is open (if
  we identify $\mathcal{PD}_{n}(\lambda)\setminus\mathcal{PT}_{n}(\lambda)$
  with a subset of $\C^{n+1}$ via the mapping $\sum_{k=0}^{n} a_{k} z^{k}
  \mapsto (a_{0},\ldots,a_{n})$). Hence, if there is a
  $f\in\mathcal{PD}_{n}(\lambda)\setminus\mathcal{PT}_{n}(\lambda)$ for which
  there is a $z_{1}\in\T$ such that
  \begin{equation*}
    \Re \frac{f(z_{1})-a_{0}}{a_{n} z_{1}^{n} - a_{0}} = \frac{1}{2},
  \end{equation*}
  then by slightly changing the cofficient $a_{1}$ (to $a_{1}^{*}$, say) we can
  obtain a polynomial $g(z) = a_{0} + a_{1}^{*} z + \sum_{k=2}^{n} a_{k}
  z^{k}\in\mathcal{PD}_{n}(\lambda)\setminus\mathcal{PT}_{n}(\lambda)$ for which
  \begin{equation*}
    \Re \frac{g(z_{1})-a_{0}}{a_{n} z_{1}^{n} - a_{0}} < \frac{1}{2},
  \end{equation*}
  and which therefore, by (\ref{eq:20}) and Theorem
  \ref{sec:an-extens-suffr-1}\ref{item:21}, cannot belong to
  $\mathcal{PD}_{n}(\lambda)$. This contradiction shows that if
  $f\in\mathcal{PD}_{n}(\lambda)\setminus\mathcal{PT}_{n}(\lambda)$ for a
  $\lambda\in[0,\frac{2\pi}{n})$, then we have
  \begin{equation}
    \label{eq:35}
    \Re \frac{f(z)-a_{0}}{a_{n} z^{n} - a_{0}}
    > \frac{1}{2}
    \quad \mbox{for} \quad z\in\C\setminus\D,
  \end{equation}
  as required.

  Suppose on the other hand that $f(z) = \sum_{k=0}^{n} a_{k} z^{k}$ with
  $|a_{0}|<|a_{n}|$ satisfies (\ref{eq:35}). Then, since $|a_{0}/a_{n}|<1$,
  all zeros of $a_{n}z^{n} - a_{0}$ lie in $\D$, and thus, by continuity,
  there must be a $\mu\in(0,\frac{2\pi}{n})$ for which (\ref{eq:21})
  holds. As shown above this is equivalent to $F_{\mu}\in
  \mathcal{D}_{n}(\mu)$ and hence to $f\in\mathcal{PD}_{n}(\mu)$.
\end{proof}

\begin{proof}[Proof of Theorem \ref{sec:class-mathc-mathc-1}]
  As shown in the proof of Theorem \ref{sec:an-extens-suffr-3}\ref{item:4}
  (cf. (\ref{eq:21})) if $F\in\mathcal{D}_{n}(\lambda)$, then
  \begin{equation*}
    \Re e^{-in\lambda/2}\frac{z \Delta_{\lambda}^{n}[F](z)}{F(e^{-i\lambda/2}z)}
    > \frac{1}{2}
    \quad \mbox{for} \quad z\in\C\setminus\D.
  \end{equation*}
  Hence, for such $F$ the polynomial $\Delta_{\lambda}^{n}[F]$ cannot vanish
  in $\C\setminus\D$ and thus all its zeros must lie in $\D$. The proof of
  the case $F\in\overline{\mathcal{D}}_{n}(\lambda)$ is similar.
\end{proof}

\begin{proof}[Proof of Theorem \ref{sec:an-extens-suffr-6}]
  We will only prove the case
  $F\in\mathcal{D}_{n}(\lambda)\setminus\mathcal{T}_{n}(\lambda)$, the case
  $F\in\overline{\mathcal{D}}_{n}(\lambda)\setminus
  \overline{\mathcal{T}}_{n}(\lambda)$ being similar.
  
  Let $P$, $Q\in\pi_{n}(\T)$ with $c_{P}=:e^{it_{P}}\neq c_{Q}=:e^{it_{Q}}$
  and suppose that
  $F=P-Q\in\mathcal{D}_{n}(\lambda)\setminus\mathcal{T}_{n}(\lambda)$. By
  Theorem \ref{sec:an-extens-suffr-4} this is equivalent to
  \begin{align*}
    e^{-it}F+ e^{it} F^{*n} = & \; e^{-it}(P -Q) + 
    e^{it}(c_{P}^{2}P-c_{Q}^{2}Q)\\ = &\; 2\left[\cos(t+t_{P}) c_{P}P -
      \cos(t+t_{Q}) c_{Q} Q\right]
        \in \mathcal{T}_{n}(\lambda)
  \end{align*}
  for all $t\in\R$. Since $e^{it_{P}}\neq e^{it_{Q}}$, the real function
  $\cos(t+t_{Q})/\cos(t+t_{P})$ takes every real value exactly once when $t$
  traverses any interval of length $2\pi$, and thus the above relation is
  equivalent to
  \begin{equation}
    \label{eq:36}
    Q\in\mathcal{T}_{n}(\lambda) \quad\mbox{and}\quad 
    c_{P}P - r c_{Q} Q \in \mathcal{T}_{n}(\lambda) \quad\mbox{for
      all}\quad r\in \R.
  \end{equation}
  By Theorem \ref{sec:class-mathc-mathc} this holds if, and only if,
  \begin{equation*}
    \Delta_{\lambda}^{n}[Q]\in\pi_{n-1}(\D) \quad\mbox{and}\quad 
    c_{P}\Delta_{\lambda}^{n}[P] - r c_{Q} \Delta_{\lambda}^{n}[Q] \in
    \pi_{n-1}(\D) 
    \quad\mbox{for all}\quad r\in \R.
  \end{equation*}
  This is true if, and only if,
  $R:=c_{P}\overline{c}_{Q}\Delta_{\lambda}^{n}[P]/ \Delta_{\lambda}^{n}[Q]$
  takes no real values in $\C\setminus\D$, i.e. if, and only if, $R$ maps
  $\C\setminus\D$ into either the upper or lower half-plane.
\end{proof}

\begin{proof}[Proof of Theorem \ref{sec:an-extens-suffr-8}]
  Note first that if $F = P-Q$ with $P,Q\in\mathcal{ST}_{n}$ and
  $c_{P}\neq c_{Q}$, then a straightforward calculation using
  (\ref{eq:55}) shows that $T = (c_{P}^{2} - c_{Q}^{2}) S$ and hence
  it suffices to consider only the polynomial $T$.
  
  By definition and the maximum principle if
  $F\in\overline{\mathcal{D}}_{n}(\lambda)$, then 
  \begin{equation*}
    I(z):=\Im e^{-in\lambda/2} \frac{F_{+}(z)}{F_{-}(z)} \geq 0,
  \end{equation*}
  or
  \begin{equation*}
    0\leq \frac{1}{i}\left(e^{-in\lambda/2}
      \frac{F_{+}}{F_{-}} - 
      e^{in\lambda/2}
      \frac{\overline{F}_{+}}{\overline{F}_{-}}\right) = 
    \frac{-i \left( F_{+}\cdot (F^{*n})_{-}-
        F_{-}\cdot(F^{*n})_{+}\right)}
    {e^{in\lambda/2}F_{-}\cdot (F^{*n})_{-}}
  \end{equation*}
  for all $z\in\T$. It is easy to check that the numerator and
  denominator of the rational function on the right-hand side of this
  expression are $2n$-self-inversive polynomials, and thus the above
  inequality is equivalent to
  \begin{equation*}
    0\leq \frac{N(t)}{D(t)} \quad \mbox{for all} \quad t\in\R,
  \end{equation*}
  where
  \begin{equation*}
    N(t) := -ie^{-int} T(e^{it})
  \end{equation*}
  and
  \begin{equation*}
    D(t) := e^{-int}e^{in\lambda/2}F_{-}(e^{it})\cdot (F^{*n})_{-}(e^{it}) = 
    \left|F_{-}(e^{it})\right|^{2}
  \end{equation*}
  are real functions by (\ref{eq:48}).

  If $F\in\overline{\mathcal{D}}_{n}(\lambda)\setminus
  \overline{\mathcal{T}}_{n}(\lambda)$, then $D(t)>0$ by Theorem
  \ref{sec:an-extens-suffr-2}, which proves (\ref{item:5}). 
  
  If, on the other hand, $F\in\pi_{n}(\D)$ and all zeros
  of $T$ on $\T$ are of even order, then either $N(t)/D(t)\geq 0$ for
  all $t\in\R$, or $N(t)/D(t)\leq 0$ for all $t\in\R$, which implies
  that we have either $I(z)\geq 0$ for all $z\in\T$, or $I(z)\leq 0$
  for all $z\in\T$. Since $F\in\pi_{n}(\D)$, the function $I$ is
  harmonic in $\C\setminus\D$ with $I(\infty)=
  \sin(n\lambda/2)>0$. The maximum principle therefore shows that
  $I(z)\geq 0$ for all $z\in\C\setminus\D$. This proves (\ref{item:1}).

  By what we have shown so far it is clear that for a
  $F\in\pi_{n}(\D)$ we have
  $F\in\mathcal{D}_{n}(\lambda)\setminus\mathcal{T}_{n}(\lambda)$ if,
  and only if, $N(t) > 0$ for all $t\in\R$, i.e. if, and only if, the
  $2n$-self-inversive polynomial $T$ does not vanish on $\T$. By
  \cite[Thm. 7.1.3]{sheil} this holds if, and only if, $T$ has exactly
  $n-1$ critical points in $\overline{\D}$. In fact, if $T$ has no
  zeros on $\T$, then $T$ cannot have any critical points on $\T$ by
  (\ref{eq:47}). The proof of the theorem is thus complete.
\end{proof}  

\begin{proof}[Proof of Theorem \ref{sec:an-extens-suffr-7}]
  Suppose first that $A\in\pi_{n}(\T)$ with $c_{A}\neq 1$ is such that
  $F(z):=A(z)-Q_{n}(\lambda;z)\in
  \overline{\mathcal{D}}_{n}(\lambda)$. Then,
  $A\in\overline{\mathcal{T}}_{n}(\lambda)$ by Lemma
  \ref{sec:furth-prop-char-2}, and since
  \begin{equation*}
    Q_{n}(\lambda;-e^{i(2j-n\pm 1)\lambda/2})=0\quad \mbox{for}\quad
    j\in\{1,\ldots,n-1\}, 
  \end{equation*}
  we have that  
  \begin{equation*}
    e^{-in\lambda/2} \frac{F_{+}(-e^{i(2j-n)\lambda/2})}{F_{-}(-e^{i(2j-n)\lambda/2})} = 
    e^{-in\lambda/2} \frac{A_{+}(-e^{i(2j-n)\lambda/2})}{A_{-}(-e^{i(2j-n)\lambda/2})}
  \end{equation*}
  is real by Lemma \ref{sec:main-result:-an-1}. Since
  \begin{equation*}
    \lim_{z\rightarrow -e^{i(2j-n)\lambda/2}} \frac{(Q_{n})_{+}(\lambda;z)}
    {(Q_{n})_{-}(\lambda;z)} = \lim_{z\rightarrow -e^{i(2j-n)\lambda/2}}
    \frac{1+e^{in\lambda/2}z}{1+e^{-in\lambda/2}z} = 
    e^{in\lambda/2}\frac{\sin\frac{j\lambda}{2}}{\sin\frac{(j-n)\lambda}{2}},
  \end{equation*}
  it follows from Theorem \ref{sec:an-extens-suffr-4} and (\ref{eq:28}) that
  \begin{equation*}
    A(-e^{i(2j-n+1)\lambda/2})=
    e^{in\lambda/2}\frac{\sin\frac{j\lambda}{2}}{\sin\frac{(j-n)\lambda}{2}}
    A(-e^{i(2j-n-1)\lambda/2})
  \end{equation*}
  for all $j\in\{1,\ldots,n-1\}$, where $A(-e^{i(2j-n-1)\lambda/2})\neq 0$
  since $A\curlyvee Q_{n}(\lambda;z)$ by Lemma
  \ref{sec:furth-prop-char-2}. Hence, if we set
  $\hat{a}:=(-i)^{n}e^{in(n-1)\lambda/4}c_{A}A(-e^{-i(n-1)\lambda/2})$, then
  $\hat{a}\in\R\setminus\{0\}$ by (\ref{eq:48}) and
  \begin{equation}
    \label{eq:37}
    A(-e^{i(2k-n-1)\lambda/2}) = \hat{a}
    \overline{c}_{A}i^{n}e^{in(2k-n-1)\lambda/4} 
    \prod_{j=1}^{k-1}\frac{\sin\frac{j\lambda}{2}}
    {\sin\frac{(j-n)\lambda}{2}}    
  \end{equation}
  for $k\in\{1,\ldots,n\}$. Moreover, $\hat{b}:=
  (-i)^{n}e^{in(n+1)\lambda/4}c_{A}A(-e^{-i(n+1)\lambda/2})\in\R$ and thus
  \begin{equation}
    \label{eq:38}
    A\left(-e^{-i(n+1)\lambda/2}\right) = \hat{b}
    \overline{c}_{A}i^{n}e^{-in(n+1)\lambda/4}.
  \end{equation}
  
  Now, suppose $P$ is as in (\ref{eq:56}) with
  \begin{equation*}
    c=\overline{c}_{A}, \quad a= \frac{\hat{a}}{2^{n}
      \prod_{j=1}^{n-1} \sin \frac{j\lambda}{2}},
    \quad b=\frac{\hat{b}}{(-2)^{n}
      \prod_{j=1}^{n} \sin \frac{(j-n-1)\lambda}{2}}.
  \end{equation*}
  Then
  \begin{equation}
    \label{eq:58}
    \begin{split}
      P\left(-e^{-i(n+1)\lambda/2}\right) =&\; b c
      Q_{n}\left(\lambda;-e^{-i(n+1)\lambda/2}\right) =
      b c \prod_{j=1}^{n}\left(1-e^{i(j-n-1)\lambda}\right)\\
      =&\; b c (-2)^{n} i^{n} e^{-in(n+1)\lambda/4}
      \prod_{j=1}^{n} \sin \frac{(j-n-1)\lambda}{2}\\
      =&\; \hat{b} \overline{c}_{A} i^{n} e^{-in(n+1)\lambda/4}.
    \end{split}
  \end{equation}
  Moreover, since
  \begin{equation*}
    \lim_{z\rightarrow -e^{i(2k-n-1)\lambda/2}}
    \frac{Q_{n}(\lambda;z)}{1+e^{i(2l-n-1)\lambda/2}z} =
    \begin{cases}
      \displaystyle \prod_{j=1\atop j\neq n+1-k} ^{n} (1-e^{i(j+k-n-1)\lambda}) \quad& 
      \mbox{if} \quad l=n+1-k,\\
      0 \quad& \mbox{otherwise}, 
    \end{cases}
  \end{equation*}
  we have, for $k\in\{1,\ldots,n\}$,
  \begin{equation}
    \label{eq:59}
    \begin{split}
      P(-e^{i(2k-n-1)\lambda/2}) = &\, 2iac \prod_{j=1\atop j\neq n+1-k} ^{n}
      (1-e^{i(j+k-n-1)\lambda}) \\
      = &\, 2(-2)^{n-1}aci^{n} e^{in(2k-n-1)\lambda/4}
      \prod_{j=1\atop j\neq n+1-k} ^{n} \sin\frac{(j+k-n-1)\lambda}{2} \\
      = &\, 2(-2)^{n-1}aci^{n} e^{in(2k-n-1)\lambda/4}
      \prod_{j=1}^{n-k}\left( -\sin\frac{j\lambda}{2}\right)
      \prod_{j=1}^{k-1}  \sin\frac{j\lambda}{2} \\
      = &\, 2^{n}aci^{n} e^{in(2k-n-1)\lambda/4}
      \prod_{j=1}^{k-1}\frac{\sin\frac{j\lambda}{2}}
      {\sin\frac{(j-n)\lambda}{2}}
      \prod_{j=1}^{n-1}  \sin\frac{j\lambda}{2} \\
      = &\, \hat{a}\overline{c}_{A}i^{n} e^{in(2k-n-1)\lambda/4}
      \prod_{j=1}^{k-1}\frac{\sin\frac{j\lambda}{2}}
      {\sin\frac{(j-n)\lambda}{2}}.
    \end{split}
  \end{equation}
  Relations (\ref{eq:37})--(\ref{eq:59}) prove that $A = P$.

  Now suppose that there are $a,b\in\R$ and $c\in\T\setminus\{\pm 1\}$ such
  that $P$ is as in (\ref{eq:56}). We will prove that in this case
  $F(z):=P(z)-Q_{n}(\lambda;z) \in\overline{\mathcal{D}}_{n}(\lambda)$ by
  employing Theorem \ref{sec:an-extens-suffr-6} and showing that
  \begin{equation*}
    \frac{\overline{c}\Delta_{\lambda}^{n}[P]}
    {\Delta_{\lambda}^{n}[Q_{n}(\lambda;z)]} = 
    \frac{\overline{c} \Delta_{\lambda}^{n}[P]}{Q_{n-1}(\lambda;z)}
  \end{equation*}
  maps $\C\setminus\overline{\D}$ into the upper or lower halfplane.
  To that end, we write
  \begin{equation*}
    H_{k}(z) = \frac{e^{i(k-n-1)\frac{\lambda}{2}}}{\sin \frac{(k-n-1)\lambda}{2}}
      \frac{1+e^{i(n+1)\frac{\lambda}{2}}z}{1+e^{i(2k-n-1)\frac{\lambda}{2}}z},
  \end{equation*}
  which means that
  \begin{equation*}
    \overline{c} P(z) = b Q_{n}(\lambda;z) + a \sum_{k=1}^{n} 
    H_{k}(z)Q_{n}(\lambda;z).
  \end{equation*}
  It is easy to check that for two rational functions $f$ and $g$
  \begin{equation*}
    \Delta_{\lambda}^{n}[f\,g] = \Delta_{\lambda}^{n}[f]\,g_{+} +  
    f_{-}\,\Delta_{\lambda}^{n}[g].
  \end{equation*}
  Hence, since $Q_{n}(\lambda;e^{i\lambda/2}z)/Q_{n-1}(\lambda;z) =
  1+e^{in\lambda/2}z$ (cf. (\ref{eq:2})),
  \begin{equation*}
    \frac{\overline{c} \Delta_{\lambda}^{n}[P]}{Q_{n-1}(\lambda;z)} = b + 
    a \sum_{k=1}^{n} \left((1+e^{in\lambda/2}z)\Delta_{\lambda}^{n}[H_{k}] 
       +  H_{k}(e^{-i\lambda/2}z) \right).
  \end{equation*}
  Straightforward computations show 
  \begin{multline*}
    (1+e^{in\lambda/2}z)\Delta_{\lambda}^{n}[H_{k}] 
    +  H_{k}(e^{-i\lambda/2}z) = \\
    \frac{e^{-i(k-n-1)\lambda/2}}{\sin \frac{(k-n-1)\lambda}{2}} + \frac{2
      i}{\sin \frac{n\lambda}{2}} \left(\frac{e^{ik\lambda/2} \sin
        \frac{(n-k)\lambda}{2}}{1+e^{i(2k-n)\lambda/2}z}+
      \frac{e^{i(k-n-1)\lambda/2} \sin
        \frac{(k-1)\lambda}{2}}{1+e^{i(2(k-1)-n)\lambda/2}z}\right),
  \end{multline*}
  and hence
  \begin{equation*}
    \frac{\overline{c} \Delta_{\lambda}^{n}[P]}{Q_{n-1}(\lambda;z)} = 
    b + a \left(\sum_{k=0}^{n-1} 
      \frac{e^{-i(k-n)\lambda/2}}{\sin \frac{(k-n)\lambda}{2}}
    + \sum_{k=1}^{n-1} 
      \frac{2i}{1+e^{i(2k-n)\lambda/2}z}\right).
  \end{equation*}
  The rational function $2i/(1+z)$ maps $\C\setminus\D$ onto 
  $\{z\in\C:\Im z \leq 1\}$ and thus we find that
  \begin{equation*}
    \Im \left(\sum_{k=0}^{n-1} 
      \frac{e^{-i(k-n)\lambda/2}}{\sin \frac{(k-n)\lambda}{2}}
    + \sum_{k=1}^{n-1} 
      \frac{2i}{1+e^{i(2k-n)\lambda/2}z}\right) \leq  -1
  \end{equation*}
  for $z\in\C\setminus\D$.  Depending on the sign of $a$,
  $\overline{c}\Delta_{\lambda}^{n}[P]/Q_{n-1}(\lambda;z)$ therefore maps
  $\C\setminus\D$ into the lower or upper half-plane.
\end{proof}

\begin{proof}[Proof of Theorem \ref{sec:an-extens-suffr-5}]
  Suppose $p(z)=\sum_{k=0}^{n}a_{k} z^{k}$
  belongs to $\mathcal{PD}_{n}(\lambda)$ for a $\lambda\in(0,\frac{2\pi}{n})$
  and satisfies $p^{*n}(0)=1$. Since $p^{*n}(z)=
  \sum_{k=0}^{n}\overline{a}_{n-k} z^{k}$, this means $a_{n}=1$, and thus it
  follows from Theorem \ref{sec:an-extens-suffr-3}\ref{item:4} that
  \begin{equation*}
    \Re \frac{p(z)-a_{0}}{z^{n}-a_{0}} >\frac{1}{2}, 
    \quad \mbox{for}\quad z\in\C\setminus\D.
  \end{equation*}
  This implies
   \begin{equation*}
     \Re
     \frac{\overline{p(\frac{1}{\overline{z}})}-
       \overline{a}_{0}}{\frac{1}{z^{n}}-\overline{a}_{0}} 
     = \Re
     \frac{p^{*n}(z)-\overline{a}_{0}z^{n}}{1-\overline{a}_{0}z^{n}} 
     >\frac{1}{2}, \quad \mbox{for}\quad z\in\overline{\D}.
  \end{equation*}
  Hence, if a function $f$ is the limit of such polynomials $p^{*n}$
  (uniformly on compact subsets of $\D$), then $\Re f(z)>\frac{1}{2}$ for all
  $z\in\D$.
\end{proof}

\begin{proof}[Proof of Theorem \ref{sec:furth-prop-char-3}]
  We will only prove the 'only if'-direction, since the other direction is
  clear. Hence, suppose that $f(z)=\sum_{k=0}^{\infty} a_{k} z^{k}$ is
  analytic in $\D$ with $f(0)=1$ and $\Re f(z) > 0$ for all $z\in\D$. Choose
  an increasing sequence $(r_{n})_{n}\subset (0,1)$ with $\lim_{n\rightarrow
    \infty} r_{n}=1$ and set, for $n\in\N$,
  \begin{equation*}
    S_{n}(z) := \sum_{k=0}^{n} a_{k} z^{k}.
  \end{equation*}
  Next, for each $n\in\N$, choose $k_{n}\in\N$ with $k_{n}>k_{n-1}$ (where
  $k_{0}=0$) such that
  \begin{equation}
    \label{eq:57}
    \Re S_{k_{n}}(r_{n}z)>0 \quad \mbox{and} \quad
    |f(r_{n}z)-S_{k_{n}}(r_{n}z)| < \frac{1}{n} \quad \mbox{for all} 
    \quad z\in\overline{\D}.
  \end{equation}
  Then, clearly, $f(z) = \lim_{n\rightarrow \infty} S_{k_{n}}(r_{n}z)$, and thus also
  \begin{equation}
    \label{eq:40}
    f(z) = \lim_{n\rightarrow \infty} S_{k_{n}}(r_{n}z) + 
    z^{k_{n}} (S_{k_{n}}(r_{n}z))^{*k_{n}}
  \end{equation}
  uniformly on compact subsets of $\D$, since by the maximum principle
  \begin{equation*}
    \left|z^{k_{n}} (S_{k_{n}}(r_{n}z))^{*k_{n}}\right| =     
    \left|z^{2k_{n}} \overline{S}_{k_{n}}(r_{n}\overline{z}^{-1})\right| <
    \left|z^{k_{n}-1} S_{k_{n}}(r_{n}z)\right|, \qquad z\in\D.
  \end{equation*}
  Set
  \begin{equation}
    \label{eq:61}
     P(z) = S_{k_{n}}(r_{n}z) + z^{k_{n}} (S_{k_{n}}(r_{n}z))^{*k_{n}}.
  \end{equation}
  Then $P$ is of degree $m:=m_{n}:=2k_{n}$ and
  \begin{align*}
    e^{-imt/2}P(e^{it}) = \,&e^{-ik_{n}t}(S_{k_{n}}(r_{n}e^{it}) + e^{ik_{n}t}
    (S_{k_{n}}(r_{n}e^{it}))^{*k_{n}})\\
    =\,& e^{-ik_{n}t}S_{k_{n}}(r_{n}e^{it}) + e^{ik_{n}t}
    \overline{S_{k_{n}}(r_{n}e^{it})} \\ =\, & 2\Re e^{-imt/2}S_{k_{n}}(r_{n}e^{it})
  \end{align*}
  for all $t\in\R$. Thus, with $t_{k}=\frac{2k\pi}{m}$,
  \begin{equation*}
    (-1)^{k}P(e^{it_{k}}) = e^{-im t_{k}}P(e^{it_{k}}) = 
    2(-1)^{k}\Re S_{k_{n}}(r_{n}e^{it_{k}}),
  \end{equation*}
  and hence, because of (\ref{eq:57}),
  \begin{equation*}
    P(e^{2k\pi i/m}) > 0 \quad \mbox{for all} \quad k=1,\ldots,m.
  \end{equation*}
  The partial fraction expansion
  \begin{equation*}
    \frac{P(z)}{1-z^{m}} = -1 + \sum_{k=1}^{m}
    \frac{P(e^{2k\pi i/m})}{m(1-e^{2k\pi i/m}z)},
  \end{equation*}
  shows that $\sum_{k=1}^{m}P(e^{2k\pi i/m})=2m$. This, in turn leads
  to,
  \begin{equation}
    \label{eq:60}
    \frac{P(z)}{1-z^{m}}
    =  \sum_{k=1}^{m}\frac{P(e^{2k\pi i/m})}{2m}\left(
      \frac{2}{1-e^{2k\pi i/m}z} - 1\right)
    =  \sum_{k=1}^{m}s_{k}^{(n)} \frac{1+e^{2k\pi i/m} z}{1-e^{2k\pi i/m}z},
  \end{equation}
  where $s_{k}^{(n)}:=\frac{P(e^{2k\pi i/m})}{2m}> 0$ for all $k=1,\ldots,m$
  and $s_{1}^{(n)} + \ldots +s_{2m}^{(n)} =1$. The assertion now follows from
  (\ref{eq:40})--(\ref{eq:60}).
\end{proof}

\begin{proof}[Proof of Theorem \ref{sec:main-results}]
  Let $\tau\in\C\setminus\{0\}$, $\gamma\in[0,1)$. We start with the proof of
  (\ref{item:9}) and thus consider a polynomial
  $P\in\pi_{n}(\overline{\Omega}_{\tau,\gamma})$ and a polynomial
  $Q\in\pi_{n}(I_{\gamma})$. Then, $\beta\in I_{\gamma}$ for every zero
  $\beta$ of $Q$, which means 
  \begin{equation}
    \label{eq:39}
    \gamma |1+\beta| < 1 - |\beta|.
  \end{equation}
  This holds if, and only if,
  \begin{equation*}
    |\beta| < \left|1 + \gamma z \left(1 +
        \beta\right)\right| \quad\mbox{for}\quad z\in \mathbb{T},
  \end{equation*}
  and hence, by the maximum principle (note that $1/(\gamma
  |1+\beta|)>1/(1-|\beta|)>1$ by (\ref{eq:39})), if, and only if,
  \begin{equation*}
    \omega(z) = \frac{-\beta z}{1 + \gamma z(1+\beta)}
  \end{equation*}
  maps $\overline{\D}$ into $\D$. Since
  \begin{equation*}
    -\beta\, w_{\tau,\gamma}(z) = \frac{-\beta \tau z}{1 + \gamma z} = 
    \frac{\tau \omega(z)}{1 + \gamma
      \omega(z)} = w_{\tau,\gamma}(\omega(z)),
  \end{equation*}
  this shows
  \begin{equation*} 
    -\beta\, \overline{\Omega}_{\tau,\gamma} = 
    -\beta\, w_{\tau,\gamma}(\overline{\D}) \subseteq w_{\tau,\gamma}(\D) 
    = \Omega_{\tau,\gamma}
  \end{equation*}
  for every zero $\beta$ of $Q$. This implies $P*_{GS}
  Q\in\pi_{n}(\Omega_{\tau,\gamma})$ by the Grace-Szeg\"o convolution theorem.

  On the other hand, our considerations show that if $Q$ of degree $n$ has a
  zero $\beta \notin I_{\gamma}$, then there is an
  $\alpha\in\overline{\Omega}_{\tau,\gamma}$ such that $-\alpha\beta \notin
  \Omega_{\tau,\gamma}$. For such an $\alpha$ the polynomial
  \begin{equation*}
    P(z):=(1-z/\alpha)^n = \sum_{k=0}^{n} {n\choose k} (-\alpha)^{-k} z^{k}
  \end{equation*}
  is of degree $n$ with all zeros in $\overline{\Omega}_{\tau,\gamma}$ and we
  have
  \begin{equation*}
    (P *_{GS} Q)(z) = Q(-z/\alpha)
  \end{equation*}
  Hence, in this case $P*_{GS} Q$ has a zero at $-\alpha \beta$ which is not
  in $\Omega_{\tau,\gamma}$. This proves Theorem
  \ref{sec:main-results}\ref{item:9} and the proof of (\ref{item:15}) is so
  similar that it can be omitted.

  If $P\in\pi_{n}(\overline{I}_{\gamma})$ and $Q\in\pi_{n}(I_{\gamma})$, then
  by (\ref{item:9}) we have $R*_{GS} Q\in\pi_{n}(\Omega_{\tau,\gamma})$ for
  all $R\in\pi_{n}(\overline{\Omega}_{\tau,\gamma})$, and consequently, by
  (\ref{item:15}), $R*_{GS} Q*_{GS} P \in\pi_{n}(\Omega_{\tau,\gamma})$ for
  all such $R$.  Another application of (\ref{item:9}) shows
  that $P*_{GS}Q\in\pi_{n}(I_{\gamma})$.  On the other hand, if $Q$ of degree
  $n$ is such that $P*_{GS}Q\in\pi_{n}(I_{\gamma})$ for all
  $P\in\pi_{n}(\overline{I}_{\gamma})$, then in particular
  \begin{equation*}
    Q(z) = (1+z)^{n} *_{GS} Q(z) \in \pi_{n}(I_{\gamma}),
  \end{equation*}
  since $-1\in\overline{I}_{\gamma}$. This proves Theorem
  \ref{sec:main-results}\ref{item:11}.

  Suppose now that $Q$ of degree $\leq n$ is such that $P*_{GS}Q\in\pi_{\leq
    n}(\C\setminus\overline{\Omega}_{\tau,\gamma})$ for all $P\in\pi_{\leq
    n}(\C\setminus\Omega_{\tau,\gamma})$. Since $P\mapsto R:=P^{*n}$ is a
  bijection between $\pi_{\leq n}(\C\setminus\Omega_{\tau,\gamma})$ and
  $\pi_{n}(\overline{\Omega}_{(\gamma^{2}-1)/\overline{\tau},\gamma})$, this
  holds if, and only if,
  \begin{equation*}
    R*_{GS}Q^{*n} = (P*_{GS}Q)^{*n} \in \pi_{n}(\Omega_{(\gamma^{2}-1)/\overline{\tau},\gamma})
  \end{equation*}
  for all
  $R\in\pi_{n}(\overline{\Omega}_{(\gamma^{2}-1)/\overline{\tau},\gamma})$.
  Because of Statement (\ref{item:9}) this is equivalent to
  $Q^{*n}\in\pi_{n}(I_{\gamma})$. Since $Q^{*n}\mapsto Q$ is a bijection
  between $\pi_{n}(I_{\gamma})$ and $\pi_{\leq n}(O_{\gamma})$, we have
  verified (\ref{item:10}). The two remaining statements of Theorem
  \ref{sec:main-results} are shown in a similar fashion, and the proof of the
  theorem is thus complete.
\end{proof}

\bibliographystyle{amsplain}

\bibliography{polyaschurlogconcave}

\end{document}